\documentclass[a4paper,12pt]{amsart}

\usepackage{amsmath}
\usepackage{amssymb}
\usepackage{mathrsfs}
\usepackage{ifthen}
\usepackage{graphicx}
\usepackage{tikz-cd}
\usepackage[T1]{fontenc} 

\def\switchlinenumbers{\@ifstar
    {\let\makeLineNumberOdd\makeLineNumberRight
     \let\makeLineNumberEven\makeLineNumberLeft}%
    {\let\makeLineNumberOdd\makeLineNumberLeft
     \let\makeLineNumberEven\makeLineNumberRight}%
    }

\def\setmakelinenumbers#1{\@ifstar
  {\let\makeLineNumberRunning#1%
   \let\makeLineNumberOdd#1%
   \let\makeLineNumberEven#1}%
  {\ifx\c@linenumber\c@runninglinenumber
      \let\makeLineNumberRunning#1%
   \else
      \let\makeLineNumberOdd#1%
      \let\makeLineNumberEven#1%
   \fi}%
  }

\nonstopmode \numberwithin{equation}{section}
\newtheorem*{theorem*}{Theorem}

\newtheorem{thm}{Theorem}[section]
\newtheorem{cor}[equation]{Corollary}
\newtheorem{lem}{Lemma}[section]
\newtheorem{prop}[equation]{Proposition}

\theoremstyle{definition}
\newtheorem{defn}{Definition}[section]
\newtheorem{example}{Example}[section]

\newtheorem{prob}[equation]{Problem}
\newtheorem{rem}{Remark}[section]

\newcounter{minutes}
\divide\time by 60
\newcounter{hours}
\multiply\time by 60
\addtocounter{minutes}{-\time}

\newcounter {own}
\def\theown {\thesection       .\arabic{own}}

\newenvironment{pf}[1][]{%
 \vskip 3mm
 \noindent
 \ifthenelse{\equal{#1}{}}%
  {{\slshape Proof. }}%
  {{\slshape #1.} }%
 }%
{\qed\bigskip}

\newcounter{alphabet}

\def\be{\begin{equation}}
\def\ee{\end{equation}}

\newcommand{\bee}{\begin{enumerate}}
\newcommand{\eee}{\end{enumerate}}

\newcommand{\blem}{\begin{lem}}
\newcommand{\elem}{\end{lem}}
\newcommand{\bthm}{\begin{thm}}
\newcommand{\ethm}{\end{thm}}
\newcommand{\bcor}{\begin{cor}}
\newcommand{\ecor}{\end{cor}}
\newcommand{\beg}{\begin{examp}}
\newcommand{\eeg}{\end{examp}}
\newcommand{\begs}{\begin{examples}}
\newcommand{\eegs}{\end{examples}}
\newcommand{\bdefe}{\begin{defin}}
\newcommand{\edefe}{\end{defin}}
\newcommand{\bprob}{\begin{prob}}
\newcommand{\eprob}{\end{prob}}
\newcommand{\bei}{\begin{itemize}}
\newcommand{\eei}{\end{itemize}}
\newcommand{\real}{{\operatorname{Re}\,}}
\newcommand{\imaginary}{{\operatorname{Im}\,}}

\newcommand{\norm}[1]{\left\lVert#1\right\rVert}
\newcommand{\abs}[1]{\left\lvert#1\right\rvert}

\begin{document}
\allowdisplaybreaks
\title{Composition operators and Carleson embeddings for the Nevanlinna class of Dirichlet series}

\author{Vasudevarao Allu}
\address{Vasudevarao Allu,
Department  of Mathematics, School of Basic Sciences,
Indian Institute of Technology Bhubaneswar,
Bhubaneswar-752050, Odisha, India.}
\email{avrao@iitbbs.ac.in}

\author{Dipon Kumar Mondal}
\address{Dipon Kumar Mondal,
Department  of Mathematics, School of Basic Sciences,
Indian Institute of Technology Bhubaneswar,
Bhubaneswar-752050, Odisha, India.}
\email{diponkumarmondal@gmail.com}

\subjclass[{AMS} Subject Classification:]{Primary 47B33, 30H10, 30H15; Secondary 30B50, 46E50.}
\keywords{Nevanlinna Class of Dirichlet series, Hardy space of Dirichlet series, Carleson measure, Composition operators.}

\def\thefootnote{}
\footnotetext{ {\tiny File:~\jobname.tex,
printed: \number\year-\number\month-\number\day,
          \thehours.\ifnum\theminutes<10{0}\fi\theminutes }
} \makeatletter\def\thefootnote{\@arabic\c@footnote}\makeatother

\begin{abstract}
This paper systematically investigates the structural, analytical, and operator-theoretic properties of the Nevanlinna class $\mathcal{N}_u$ of Dirichlet series, introduced by Brevig and Perfekt [Adv.\ Math., 2021] and further developed by Guo \textit{et al.}\ [Ann.\ Inst.\ Fourier (Grenoble), 2025]. First, we examine the topological structure of $\mathcal{N}_u$. We then prove a Littlewood--Paley type identity for functions in $\mathcal{N}_u$, establish an equivalent characterization via vertical limit functions, and demonstrate that the interchange of limits in this identity is permissible. In addition, we provide an alternative proof of this identity using potential-theoretic approach. Applying these analytical tools, we study composition operators $C_\Phi$ acting on $\mathcal{N}_u$ and characterize those symbols $\Phi$ that induce bounded composition operators $C_\Phi$. Utilizing Carleson measure techniques on half-planes and infinite-dimensional tori, we characterize the symbols $\Phi$ that generate bounded and compact composition operators. In particular, we prove the complete equivalence between the (vanishing) Carleson embedding condition and the geometric (vanishing) Carleson condition on Carleson squares. We conclude with two function-theoretic applications to functions in $\mathcal{N}_u$.
\end{abstract}

\maketitle
\pagestyle{myheadings}
\markboth{Vasudevarao Allu and Dipon Kumar Mondal}{Nevanlinna class of Dirichlet series}

\section{Introduction}

The theory of Nevanlinna classes exhibits a rich duality structure, driven by its unusual topological and metric properties (see, for instance, \cite{Haldimann-Studia-2001, ShapiroShields-Amer-1975}). These classes form a prominent family of topological algebras, collectively known as Nevanlinna algebras. In the present paper, we study the Dirichlet series analogues of classical Nevanlinna spaces on the open unit disk $\mathbb{D}$.
 A Dirichlet series is a series of function of the form $$f(s) = \sum_{n=1}^{\infty}a_nn^{-s},$$ where \(s = \sigma + it\) is a complex number with real part $\real(s)=\sigma.$ 
We denote by $\mathcal{H}^2$ the Hardy space of Dirichlet series consisting of functions $f$ whose coefficients are square summable, that is,
    \begin{equation}\label{Dipon-vasu-p2-eqn-001}
    	\norm{f}^2_{\mathcal{H}^2} = \sum_{n=1}^{\infty}|a_n|^2<\infty.
    \end{equation}
    Furthermore, let $\mathcal{H}^{\infty}$ denote the set of all bounded analytic functions in the right half-plane which admit a representation as a Dirichlet series converging in some half-plane $\mathbb{C}_\theta := \{s=\sigma+ it : \real(s) > \theta\}$ for a given $\theta > 0$. The space $\mathcal{H}^{\infty}$ is equipped with the supremum norm. A classical theorem of Harald Bohr \cite{Bohr-Angew-1913} asserts that if $f$ belongs to $\mathcal{H}^{\infty}$, then the associated Dirichlet series converges uniformly to $f$ in $\mathbb{C}_\theta$ for every $\theta > 0$. In particular, if $f$ satisfies $\sigma_u(f)\le 0$, then $f$ is almost periodic on $\overline{\mathbb{C}_\theta}$ for every $\theta>0$. A standard reference for the theory of almost periodic functions is due to Besicovitch \cite{Besicovitch-AP-Book}. The abscissa of uniform convergence of a Dirichlet series $f$, denoted by $\sigma_u(f)$, is defined by

\begin{equation*}
		\sigma_u(f): = \inf\bigg\{\theta\in\mathbb{R}: f(s)= \sum_{n=1}^{\infty}a_nn^{-s}\,\,\,\mbox{converges uniformly on}\,\,\, {\mathbb{C}_\theta}\bigg\}\in\overline{\mathbb{R}},
\end{equation*}
where $\overline{\mathbb{R}} = \mathbb{R}\cup\{\pm\infty\}$ is the extended real line. Also, let $\sigma_b(f)$ denote the infimum of all $\theta$ in $\mathbb{R}$ such that $f$ may be extended (by analytic continuation if necessary) to a bounded analytic function in $\mathbb{C}_\theta$. Bohr \cite{Bohr-Angew-1913} proved that $\sigma_u(f)=\sigma_b(f).$ Let $\mathcal{P}_D$ denote the set of all Dirichlet polynomials of the form $P(s) = \sum_{n=1}^{N}a_nn^{-s}.$ For $0<p<\infty$ and $P\in\mathcal{P}_D$, it follows from the almost periodicity of the function $t\mapsto |P(it)|^p$ that the limit
\begin{equation*}
	\norm{P}^p_{\mathcal{H}^p}= \lim_{T\to\infty}\frac{1}{2T}\int_{-T}^{T}|P( it)|^pdt
\end{equation*}
exists (see for instance \cite{Queffelec-TRM-2020}). The Hardy space of Dirichlet series $\mathcal{H}^p$ is defined to be the completion of $\mathcal{P}_D$ in the metric $\norm{.}^p_{\mathcal{H}^p}$. \\[0.5mm]

The classical Nevanlinna class $N(\mathbb{D})$  is well studied (see, e.g., \cite{ChoaKim-PAMS-1997,Duren-H^p space-1970,Masri-Thesis-1985}). By contrast, its analogue for Dirichlet series has received little attention. Recently, Guo \textit{et al.} \cite{GuoZhou-AnnalFourier-2025} introduced the Nevanlinna class on the infinite-dimensional polydisk. It is well-known that functions of infinitely many variables and Dirichlet series are intimately linked via the Bohr correspondence. In connection with composition operators on Dirichlet series spaces, Brevig and Perfekt \cite{Brevig-AdvMath-2021} introduced the following notion of the Nevanlinna class for a Dirichlet series $f$.

 \begin{defn}\label{Dipon-vasu-p2-defn-02}
	A Dirichlet series $f$ with $\sigma_u(f)\le 0$ is defined in the \textit{Nevanlinna class of Dirichlet series} $\mathcal{N}_u$ if the following condition holds:
	\begin{equation}\label{Dipon-vasu-p2-eqn-005}
		{\limsup_{\sigma\to 0^+}}\lim_{T\to\infty}\frac{1}{2T}\int_{-T}^{T}\log^+|f(\sigma +it)|\,dt<\infty,
	\end{equation}
	where $\log^+(f) = \max\{\log(f), 0\}$.
\end{defn}
 If $\sigma_u(f) \le 0$ and $f \in \mathcal{H}^p$, then $f \in \mathcal{N}_u$, which follows from the elementary inequality $\log^+ t \le \frac{1}{p}\, t^p$ (valid for $t \ge 0$ and $p > 0$). We emphasize that the condition $\sigma_u(f) \le 0$ is essential; without it, the class $\mathcal{N}_u$ cannot be properly defined. For example, consider the Dirichlet series$$g(s) = \sum_{n=2}^\infty \frac{n^{-s}}{\sqrt{n} \log n}.$$It is easy to see that $g \in \mathcal{H}^2$, whereas $\sigma_u(g) = 1/2$.
 	The inequalities  
 	\begin{equation}\label{Dipon-vasu-p2-eqn-044}
 		\log^+(x)\le\log(1+x)\le \log(2)+\log^{+}(x),\,\,\, x\ge 0
 	\end{equation}
 	  along with \eqref{Dipon-vasu-p2-eqn-005} ensures that	$f\in \mathcal{N}_u$ if, and only if,
 \begin{equation}\label{Dipon-vasu-p2-eqn-002}
 	\norm{f}_{0}:= \limsup_{\sigma\to 0^+}\lim_{T\to\infty}\frac{1}{2T}\int_{-T}^{T}\log(1+|f(\sigma+it)|)\,dt <\infty,\,\,\,\, f\in\mathcal{N}_u.
 \end{equation}
 Here, the class $\mathcal{N}_u$ is an $F$-space; that is, a topological vector space whose topology is induced by a complete, translation-invariant metric defined by
 \begin{equation*}
 	d_0(f,g) = \norm{f-g}_0 \quad \text{for } f,g \in \mathcal{N}_u.
 \end{equation*}
 Moreover, $\mathcal{N}_u$ forms an $F$-algebra that is neither locally bounded nor locally convex. Although it is not a normed algebra, several authors denote this quantity by $\vert\vert\vert \cdot \vert\vert\vert_0$. To maintain standard notation without risk of confusion, we adopt the norm-like notation $\Vert{}\cdot\Vert{}_0$ as defined in \eqref{Dipon-vasu-p2-eqn-002}. For further algebraic and topological properties of this space, we refer the reader to \cite{GuoZhou-AnnalFourier-2025, Haldimann-Studia-2001, ShapiroShields-Amer-1975}.
 Since $f$ is a Dirichlet series with $\sigma_u(f) \le 0$ and $h(x) = \log(1+x)$ is continuous and non-decreasing on $[0, \infty)$, the composition $t \mapsto \log(1+\vert{}f(it)\vert{})$ is almost periodic (see, e.g., \cite{Besicovitch-AP-Book,Queffelec-TRM-2020}). Consequently, the limit \eqref{Dipon-vasu-p2-eqn-002} exists.
In contrast, Guo \textit{et al.} \cite{GuoZhou-AnnalFourier-2025} defined the Nevanlinna--Dirichlet class $\mathcal{N}$ as the completion of $\mathcal{N}_u$ under the metric $\Vert{}\cdot\Vert{}_0$ given in \eqref{Dipon-vasu-p2-eqn-002}. Under the canonical embedding of \(\mathcal N_u\) into \(\mathcal N\), the topology induced on \(\mathcal N_u\) by the completion coincides with the topology induced by \(\|\cdot\|_0\). \\[0.5mm]

 In function theory, the Littlewood–Paley identity serves as a fundamental tool for establishing the boundedness and compactness of operators on a given function space. However, the derivation of such identities is not often straightforward. The primary objective of this article is to develop a Littlewood–Paley type identity for Dirichlet series of the form \( f(s) = \sum_{n=1}^{\infty} a_n n^{-s} \) in relation to the norm \( \|f\|_0 \).
\begin{thm}\label{Dipon-vasu-p2-thm-01}
	Suppose the Dirichlet series $f$ is in $\mathcal{N}_u$ satisfying $\sigma_u(f)\le 0$ and $f\not\equiv 0$. Then 
	\begin{equation}\label{Dipon-vasu-p2-eqn-004}
		\norm{f}_{0}= \log(1+ |f(+\infty)|) + \lim_{{\sigma_0}\to 0^+}\lim_{T\to\infty}\frac{1}{2T}\int_{{\sigma_0}}^{\infty}\int_{-T}^{T}\frac{|f'(\sigma+it)|^2}{|f(\sigma+it)|(1+|f(\sigma+it)|)^2}\,(\sigma-{\sigma_0}) \,dt\, d\sigma.
	\end{equation}
\end{thm}
The proof of Theorem \ref{Dipon-vasu-p2-thm-01} is based on the recent work of Brevig \textit{et al.} \cite{BrevigKouro-TAMS-2025}, not only provides a different perspective but also leads naturally to a few interesting questions. An observation due to Jessen \cite{Jessen-acta-1945} asserts that if a Dirichlet series \( f \) belongs to \( \mathcal{H}^{\infty} \), then for every $\sigma>0$, the limit
\begin{equation}\label{Dipon-vasu-p2-eqn-011}
	\mathcal{A}(f, \sigma) := \lim_{T\to\infty}\frac{1}{2T}\int_{-T}^{T}\log(1+|f(\sigma+it)|)\,dt
\end{equation}
exists. The existence of the limit in \eqref{Dipon-vasu-p2-eqn-011} also follows from \cite[Theorem 3.8]{Brevig-AdvMath-2021} relying on the fact that the function $w\mapsto\mathcal{A}(f,w)$ is convex and non-increasing. In analogy with the corresponding definition of the Nevanlinna class of Dirichlet series $\mathcal{N}_u$ with $\sigma_u(f)\le 0$, we define
\begin{equation*}
	\norm{f}_0:= \limsup_{\sigma \to 0^+}\mathcal{A}(f, \sigma).
\end{equation*}

In $2009$, Saksman and Seip \cite{Saksman-BLMS-2009} showed that, in general, the limits in $\sigma$ and $t$ appearing in Carlson's theorem \cite[Lemma 3.3]{Hedenmalm-Duke-1997} cannot be interchanged. Motivated by this observation, Brevig \textit{et al.} \cite{BrevigKouro-TAMS-2025} proved that such an interchange is permissible in the Littlewood–Paley formula for the Hardy spaces of Dirichlet series $\mathcal{H}^p$, for $1\le p<\infty.$ In light of these results, it is natural to investigate whether an analogous phenomenon holds in our setting, as described in Theorem \ref{Dipon-vasu-p2-thm-01}.

\begin{thm}\label{Dipon-vasu-p2-thm-04}
	Let the Dirichlet series $f$ belongs to $\mathcal{N}_u$ satisfying $\sigma_u(f)\le 0$ and $f\not\equiv 0$. Then
	\begin{align*}
		\small\lim_{T\to\infty}\bigg|\frac{1}{2T}\int_{-T}^{T}\log(1+|f(it)|)\,dt\, &-\, \log(1+|f(+\infty)|)\, \\&-\, \frac{1}{2T}\int_{0}^{\infty}\int_{-T}^{T}\frac{|f'(\sigma+it)|^2}{|f(\sigma+it)|(1+|f(\sigma+it)|)^2}\sigma\,dt\, d\sigma\bigg| = 0.
	\end{align*}
\end{thm}
An analogous phenomenon holds for the mean counting function
\begin{equation}\label{Dipon-vasu-p2-eqn-016}
	M_f(w) = \lim_{\sigma_0\to 0^+}\lim_{T\to\infty}\frac{\pi}{T}\sum_{\substack{s\in f^{-1}(\{w\})\\|\imaginary(s)|<T\\\sigma_0<\real(s)<\infty}} (\real(s)-\sigma_0)
\end{equation}
exists for every $f\in\mathbb{C}\setminus\{f(+\infty)\}$ (see, for instance, \cite[Theorem 1.5]{BrevigKouro-TAMS-2025}). This mean counting function was first introduced by Brevig and perfekt \cite{Brevig-AdvMath-2021} as an analogue of the Nevanlinna counting function.\\

Now assume $\Phi$ be a holomorphic self-map of a domain $\Omega$ in the complex plane $\mathbb{C}.$ Then the composition $C_\Phi(f) = f\circ\Phi$ defines a composition operator $C_\Phi$ on the space of holomorphic function on $\Omega.$
In $1999$, Gordon and Hedenmalm \cite{Gordon-Michigan-1999} characterized those holomorphic functions $\Phi$ defined on a right half-plane $\mathbb{C}_\theta$ $(\theta\in\mathbb{R})$ that maps into $\mathbb{C}_{{1}/{2}}$, which generates a composition operator $C_\Phi$ on $\mathcal{D}$. Moreover, Gordon and Hedenmalm \cite{Gordon-Michigan-1999} also determined precisely for which symbol $\Phi$ generates a bounded composition operator on $\mathcal{H}^2$. They proved that the symbol $\Phi$ defined on the half-plane $\mathbb{C}_{{1}/{2}}$ generates a composition operator if, and only if,
\begin{equation*}
	\Phi(s) = c_0s+ \phi(s),
\end{equation*} where $c_0$ is a non-negative integer and $\phi(s)= \sum_{n=1}^{\infty}c_nn^{-s}$ converges uniformly in each half-plane $\mathbb{C}_\epsilon$, for every $\epsilon>0$, where $\phi$ satisfies the following properties:
\begin{itemize}
	\item[\textbf{(i)}] If $c_0=0$, then $\Phi(\mathbb{C}_0)= \phi(\mathbb{C}_0)\subset\mathbb{C}_{{1}/{2}}$.
	\item[\textbf{(ii)}] If $c_0\ge 1$, then either $\phi= 0$ or, simply $\phi(\mathbb{C}_0)\subset\mathbb{C}_0$.
\end{itemize}
The non-negative integer $c_0$ is called the characteristic of $\Phi$ or, simply \textit{{char} $\Phi$}. We shall denote $\mathcal{G}_0$ and $\mathcal{G}_{\ge 1}$ for the subclasses \textbf{(i)} and \textbf{(ii)} respectively. These two classes are collectively referred to as the Gordon--Hedenmalm class of symbols $\Phi$. A complete characterizations of the symbol $\Phi$ in the G\"{o}rdon-Hedenmalm class $\mathcal{G}_0$ that induce compact composition operators on $\mathcal{H}^2$ has been obtained by several authors (see for instance \cite{Bayart-monatsh-2002, Bayart-IJ-2003, Bayart-BLMS-2023, Athanosios-Revista-2024, Brevig-AdvMath-2021}). In contrast, the classification of symbols $\Phi \in \mathcal{G}_{\ge 1}$ that generates compact composition operators on $\mathcal{H}^2$ remains an active area of research. Moreover, establishing the necessary conditions for the compactness of $C_\Phi$ on $\mathcal{H}^2$, without assuming additional conditions, is still an open problem.\\

Motivated by the seminal work of Gordon and Hedenmalm \cite{Gordon-Michigan-1999} on composition operators on the Hardy space $\mathcal{H}^2$ of Dirichlet series, and its subsequent extension to the $\mathcal{H}^p$ setting by Bayart \cite{Bayart-monatsh-2002}, it is natural to investigate which analytic symbols $\Phi$ generates bounded composition operators $C_\Phi$ on the Dirichlet--Nevanlinna class $\mathcal{N}_u$.

\begin{thm}\label{Dipon-vasu-p2-thm-09}
	Let $\Phi : \mathbb{C}_1 \to \mathbb{C}_1$ be a symbol with char$(\Phi)=0$. 
	If $C_\Phi$ defines a bounded composition operator on $\mathcal{N}_u$ with $\sigma_u(f)\le 0$, then $\Phi$ extends holomorphically to a mapping from $\mathbb{C}_0$ into $\mathbb{C}_0$. Conversely, if $\Phi$ extends holomorphically to a mapping from $\mathbb{C}_0$ into $\mathbb{C}_0$, then $C_\Phi$ defines a bounded composition operator on $\mathcal{N}_u$ satisfying $\sigma_u(f)\le 0$.
\end{thm}

 The study of composition operators on the classical Nevanlinna class $N(\mathbb{D})$, defined on the unit disk $\mathbb{D} := \{z\in\mathbb{C} : |z| < 1\},$ was initiated by Masri in his thesis \cite{Masri-Thesis-1985}. Subsequently, Choa and Kim \cite{ChoaKim-PAMS-1997} proved that a composition operator $C_\Phi$ is compact on $N(\mathbb{D})$ if, and only if, it is compact on the classical Hardy space $\mathbb{H}^2(\mathbb{D}).$ In this paper, we develop a Carleson measure approach to determine the compactness of composition operators acting on the Nevanlinna class of Dirichlet series $\mathcal{N}_u$. Recall that on the classical Hardy spaces $H^p(\mathbb{D})$, compactness was characterized by Shapiro \cite{Shapiro-annals-1987} via the Nevanlinna counting function and shown to be independent of $p>0$. A key ingredient in this circle of ideas originates from  Carleson \cite{Carleson-Annals-1962}, who proved that the class of Carleson measures for $H^p(\mathbb{D})$ is identical for all $p > 0$. In the present setting, the relevant Carleson measures are precisely the push-forward measures generated by symbols $\Phi \in \mathcal{G}_0$. A positive Borel measure $\mu$ is a Carleson measure for the Hardy space of Dirichlet series $\mathcal{H}^p$ for $p\in [1,\infty)$, if there exists a constant $C>0$ such that for all $f\in\mathcal{H}^p$ it holds that
\begin{equation}\label{Dipon-vasu-p2-eqn-040}
	\int_{\mathbb{C}_{1/2}}|f(s)|^p\,d\mu(s) \le C\norm{f}^p_{\mathcal{H}^p}.
\end{equation}

The smallest possible $C$ in \eqref{Dipon-vasu-p2-eqn-040} is called the Carleson norm of $\mu$ with respect to $\mathcal{H}^p$, denoted by $\Vert{}\mu\Vert{}_{C,\mathcal{H}^p}$. If $\mu$ is not a Carleson measure for $\mathcal{H}^p$, we set $\Vert{}\mu\Vert{}_{C,\mathcal{H}^p} = \infty$. Carleson measures for the Nevanlinna class $\mathcal{N}_u$ are defined analogously, following the classical setting (see, for instance, \cite[Section 4]{ChoeKoo-JdeAnalyse-2008}). A precise discussion is provided in Section \ref{Dipon-vasu-p2-sec-2.4}. We now formulate geometric criteria for Carleson measures on $\mathcal{N}_u$ in terms of Carleson squares on the boundary of the half-plane $\mathbb{C}_0$. These geometric conditions will be central to characterize the boundedness and compactness of the composition operator $C_\Phi$.

\begin{thm}\label{Dipon-vasu-p2-thm-05}
	Let $\Phi$ be a Dirichlet series inducing a holomorphic self-map of $\mathbb{C}_0$ with ${char}(\Phi) = 0$, and suppose that $\overline{\Phi(\mathbb{C}_0)}$ is a compact subset of $\mathbb{C}_0$. Then the following assertions are equivalent:
	\begin{enumerate}
		\item[\rm (i)] The composition operator $C_\Phi$ is bounded on $\mathcal{N}_u$ with $\sigma_u(f)\le 0$ if, and only if, the push-forward measure $\mu_{\Phi^*}$ satisfies the geometric Carleson condition on $\mathbb{C}_0$.
		\item[\rm (ii)] The composition operator $C_\Phi$ is compact on $\mathcal{N}_u$ with $\sigma_u(f)\le 0$ if, and only if, the push-forward measure $\mu_{\Phi^*}$ satisfies the geometric vanishing Carleson condition on $\mathbb{C}_0$.
	\end{enumerate}
\end{thm}

We refer to Section~\ref{Dipon-vasu-p2-sec-2.4} for the precise geometric formulations of the Carleson and vanishing Carleson conditions on $\mathbb{C}_0$.
The proof of Theorem \ref{Dipon-vasu-p2-thm-05} relies on the original technique introduced by Carleson \cite{Carleson-Annals-1962}, which has since been adapted and extended by numerous authors. Under the standing assumption
$\overline{\Phi(\mathbb C_0)}$ is a compact subset of $\mathbb C_0$, the measure $\mu_{\Phi^*}$ is automatically vanishing Carleson. Hence using the
vanishing version of the Carleson embedding theorem, assertion~(ii)
follows from the proof of~(i).\\[0.5mm]

The classical Nevanlinna class \( N(\mathbb{D}) \) has been extensively studied by several authors (see for instance \cite{ChoaKim-PAMS-1997, Masri-Thesis-1985, ShapiroShields-Amer-1975}). In contrast, comparatively less attention has been devoted to its Dirichlet series analogue $\mathcal{N}_u$. Building on the perspective of Shapiro and Shields  \cite{ShapiroShields-Amer-1975}, the Nevanlinna class of Dirichlet series \(\mathcal{N}_u\), exhibits a distinctive topological property that sets it apart from the Hardy spaces of Dirichlet series \(\mathcal{H}^p\). While \(\mathcal{N}_u\) forms a topological group under addition, it fails to be a topological linear space, as scalar multiplication is not continuous. This fact is evident from the characteristics of the class \(N(\mathbb{D}^{\infty}_1)\) described by Guo \textit{et al.} \cite{GuoZhou-AnnalFourier-2025}, where $\mathbb{D}^{\infty}_1 := \ell^1 \cap \mathbb{D}^\infty$ is a domain in the Banach space $\ell^1$ of absolutely summable sequences. In the one-variable setting, Shapiro and Shields \cite{ShapiroShields-Amer-1975} demonstrated that the classical Nevanlinna class $N(\mathbb{D})$ is disconnected. In light of this finding, we establish that our class $\mathcal{N}$ as well as $\mathcal{N}_u$ is likewise disconnected. Our approach is based on the structure of the class $N(\mathbb{D}^{\infty}_1)$, along with the isometric correspondence between $\mathcal{N}$ and $N(\mathbb{D}^{\infty}_1)$ provided by Bohr's transform.\\[0.5mm]



	The purpose of this paper is to develop a composition operator theory for the Nevanlinna class of Dirichlet series. While composition operators have been extensively studied on Hardy and Bergman spaces of Dirichlet series, the Nevanlinna class is governed by a fundamentally different logarithmic mean functional rather than an $L^p$-type norm. In particular, the natural gauge of the Nevanlinna class of Dirichlet series is not induced by a Banach norm, and the corresponding boundary theory is governed by logarithmic rather than power integrability. We first establish a Hardy-Stein-type identity for the $\mathcal{N}_u$ gauge, expressing its logarithmic size in terms of a weighted Dirichlet energy. We then prove a complete characterisation of bounded composition operators generated by characteristic-zero Dirichlet series symbols: boundedness is equivalent to the holomorphic extension of the symbol to a self-map of the boundary half-plane. The necessity is obtained via a natural-boundary argument, while sufficiency follows from a logarithmic composition estimate. Our second main result identifies the associated boundary pushforward measure as the natural Carleson measure governing composition. This yields a necessary and sufficient Carleson embedding criterion for boundedness and, in turn, a vanishing Carleson characterisation of compactness. Thus, this paper establishes a three-way correspondence between the boundary distribution of the symbol, logarithmic Carleson embeddings, and composition operators on the Nevanlinna class of Dirichlet series.\\[0.5mm]

The organization of this paper as follows. Section~\ref{Dipon-vasu-p2-sec-2} contains essential preliminaries and background material. Section~\ref{Dipon-vasu-p2-sec-3} begins by demonstrating that $\mathcal{N}$ is disconnected. We then establish a Littlewood–Paley type identity and its counterpart via vertical limit functions, proving that the interchange of limits is permissible on $\mathcal{N}_u$. We also present an alternative, potential-theoretic approach to establish Theorem~\ref{Dipon-vasu-p2-thm-01}. In Section~\ref{Dipon-vasu-p2-sec-4}, we prove Theorem \ref{Dipon-vasu-p2-thm-09} that characterizes the class of symbols $\Phi$ that induce a bounded composition operator $C_\Phi$ on $\mathcal{N}_u$. In particular, we prove Theorem~\ref{Dipon-vasu-p2-thm-05} via the equivalence between the (vanishing) Carleson embedding condition and the geometric (vanishing) Carleson measure condition on Carleson squares. Section~\ref{Dipon-vasu-p2-sec-5} concludes with two applications of the Nevanlinna class $\mathcal{N}_u$: the first establishes a connection with hyperbolic spherical derivatives, while the second investigates the invertible elements of $\mathcal{N}_u$.

	\vspace{2mm} 
	\paragraph{\textbf{Notation}} Throughout this article, $C$ denotes a positive constant which may vary from line to line. When necessary, we write that $C=C(x)$ to indicate that the constant depends on a parameter $x.$ If $f$ and $g$ are two real functions defined on the same set $S$, the notation $f \asymp g$ means that there exists a constant $C\ge 1$ such that $C^{-1}f(x) \le g(x) \le Cf(x)$, for all $x$, while $f(x)\lesssim g(x)$ signifies that there exists a constant $C>0$ such that $f(x)\le Cg(x)$, for all $x$. Finally, for $w\in\mathbb{C}$ and $r>0$, we denote by  $\mathbb{D}(w,r):=\{\xi\in\mathbb{C}: |\xi-w|<r\}$ the open disk centered at $w$ and radius $r$.

\section{Preliminaries}\label{Dipon-vasu-p2-sec-2}

 Hardy spaces in infinitely many variables are closely related to spaces formed by Dirichlet series. In recent years, the function theory in infinitely many variables has received significant attention (see for instance\cite{Aleman-IMRN-2013, Gamelin-PLMS-1986, GuoZhou-AnnalFourier-2025}). 
 \subsection{Hardy space of infinite polytorus and vertical limits} For \(0 < p < \infty\), the Hardy space of the infinite polytorus, denoted by \(\mathbb{H}^p(\mathbb{T}^{\infty})\), is defined  as the closure of all polynomials \(\mathcal{P}_\infty\) in \(L^p(\mathbb{T}^{\infty})\). Here \(\mathbb{T}^\infty = \{\chi = (\chi_1, \chi_2, \ldots): |\chi_i| = 1 \text{ for } i \in \mathbb{N}\}\) denotes the countable infinite Cartesian product of the unit circle \(\mathbb{T}\), equipped with the normalized Haar measure \(m_\infty\) of infinite polytorus with the countable infinite product measure $m \times m \times \cdots,$ where $m$ is the normalized Lebesgue measure of the unit circle. A polynomial, in this context, is understood to be an analytic polynomial depending on finitely many complex variables. For further properties of $\mathbb{H}^p(\mathbb{T}^{\infty})$, we refer to \cite{Bayart-monatsh-2002, Gamelin-PLMS-1986, GuoZhou-AnnalFourier-2025}. Similarly, the Hardy space of Dirichlet series $\mathcal{H}^p$ (for $p > 0$) is defined to be the completion of the set of Dirichlet polynomials $\mathcal{P}_D$ under the norm $\Vert{}\cdot\Vert{}_{\mathcal{H}^p}$, given by
 \begin{equation*}
 	\norm{P}_{\mathcal{H}^p}^p = \lim_{T\to\infty}\frac{1}{2T}\int_{-T}^{T}|P(it)|^p\,dt,
 \end{equation*}
 where $P(s) = \sum_{n=1}^{N}a_n n^{-s}$ is a Dirichlet polynomial.
    Bohr's \cite{Bohr-Acta-1913} vision identifies $\mathcal{H}^p$ with the Hardy space of infinite polytorus $\mathbb{H}^p(\mathbb{T}^{\infty})$. More precisely, the Bohr's correspondence $\mathcal{B}$ defines an algebraic isomorphism from $\mathcal{P}_D$ onto the space of analytic polynomials $\mathcal{P}_\infty$. Furthermore, by the Birkhoff-Oxtoby theorem \cite{Queffelec-TRM-2020}, for every  $P\in\mathcal{P}_D$, we obtain
    \begin{equation*}
    	\norm{P}^p_{\mathcal{H}^p}= \int_{\mathbb{T}^{\infty}}|\mathcal{B}P|^pdm_\infty.
    \end{equation*}
    Consequently, the Bohr's correspondence extends uniquely to an isometric isomorphism from $\mathcal{H}^p$ onto $\mathbb{H}^p(\mathbb{T}^{\infty})$. \\[0.5mm]
    
    For $\tau\in\mathbb{R}$, let $T_\tau$ denote the vertical translations $T_\tau f(s) = f(s+i\tau).$
    By the uniqueness of prime factorization, the infinite-dimensional torus $\mathbb{T}^\infty$ can be identified with the group of characters on the group $(\mathbb{Q}_+, \cdot)$. For a point $\chi = (\chi_1, \chi_2, \ldots) \in \mathbb{T}^\infty$, the associated character $\chi : \mathbb{Q}_+ \to \mathbb{T}$
    is defined as the completely multiplicative function on $\mathbb{N}$ satisfying
    $\chi(p_j) = \chi_j,$ where $\{p_j\}_{j \ge 1}$ denotes the increasing sequence of prime numbers. This definition extends naturally to all of $\mathbb{Q}_+$ by imposing the relation $\chi(n^{-1}) = (\chi(n))^{-1}.$ For a Dirichlet series $f(s) =\sum_{n\ge1} a_nn^{-s}$ and a character $\chi$, the vertical limit function $f_\chi$ is defined as, $f_\chi(s) = \sum_{n\ge1} a_n\chi(n)n^{-s}$. The vertical translations $T_\tau(s)$ corresponds to the vertical limit function $f_\chi$ of the characters $\chi(n) = n^{-i\tau}.$
    The vertical limit functions $f_\chi$ are precisely the functions obtained by uniform limits $$f_\chi(s) = \lim_{j\to\infty} T_{\tau_j}f(s) = \lim_{j\to\infty} f(s+i\tau_j)$$ in $\overline{\mathbb{C}_\theta}$, where $(\tau_j)_{j\ge 1}$ is a sequence of real numbers.
     As noted by Bayart \cite{Bayart-monatsh-2002}, examining $f_\chi$ provides a clearer understanding of the function-theoretic behavior of $f$. In particular, for almost every $\chi$, $f_\chi$ can be analytically continued to the half-plane $\mathbb{C}_0$. 

    \subsection{Dirichlet series analogue of Nevanlinna class}  In connection with composition operators on Dirichlet series spaces, Brevig and Perfekt \cite{Brevig-AdvMath-2021} introduced the notion of the Nevanlinna class for a Dirichlet series $f$ defined in the definition \eqref{Dipon-vasu-p2-defn-02}. Motivated by this Guo \textit{et al.} \cite{GuoZhou-AnnalFourier-2025} defined the class $\mathcal{N}_u$ consists of all Dirichlet series $f$ with \( \sigma_u(f) \le 0 \) and $$ \limsup_{\sigma \to 0^+} \| f_\sigma \|_0 < \infty,$$ where \( f_\sigma(s) = f(s+\sigma) \) for $\sigma\in\mathbb{R}.$ Moreover, they proved for a Dirichlet series $f$ with $\sigma_u(f)\le 0$, $\norm{f_\sigma}_0$ defines a non-increasing function of $\sigma$. Hence we can write
    \begin{equation*}
    	\lim_{{\sigma}\to 0^+}\norm{f_\sigma}_0 = \sup_{\sigma>0}\norm{f_\sigma}_0.
    \end{equation*}
      For $f$ on the dense subset $\mathcal{N}_u,$ they introduced the norm of a Dirichlet series $f$ in the class $\mathcal{N}_u$ as 
		 \begin{equation}\label{Dipon-vasu-p2-eqn-006}
		 	\norm{f}_0 = \sup_{\sigma>0}\norm{f_\sigma}_0.
		 \end{equation}
		 It defines a metric on the dense subset $\mathcal{N}_u$ given by
		 \begin{equation*}
		 	d_0(f,g) = \norm{f - g}_0,\,\,\,\mbox{for}\,\, f,g\in\mathcal{N}_u.
		 \end{equation*}
		 The norm defined in \eqref{Dipon-vasu-p2-eqn-006} coincides with the definition \eqref{Dipon-vasu-p2-defn-02} introduced by Brevig and Perfekt \cite{Brevig-AdvMath-2021}. But the Nevanlinna-Dirichlet class of Dirichlet series \( \mathcal{N} \) is defined as the completion of \( \mathcal{N}_u \) with respect to the metric $d_0$, that is $$\mathcal{N} = \overline{\mathcal{N}_u}^{d_0}.$$ Thus 
		 \begin{equation*}
		 	d_\mathcal{N}(f,g) = d_0(f,g),\,\,\,\mbox{for all}\,\,\, f, g\in\mathcal{N}_u.
		 \end{equation*}
		 In conclusion, for any $0<p<q<\infty,$ the following inclusions hold true
		 \begin{equation*}
		 	\mathcal{H}^\infty\subset\mathcal{H}^q\subset\mathcal{H}^p\subset\mathcal{N}_u.
		 \end{equation*}
		 
      \subsection{Nevanlinna class in infinitely many variables} Recently, Guo \textit{et al.} \cite{GuoZhou-AnnalFourier-2025}  explored the Nevanlinna class in the setting of the infinite-dimensional polydisk.
      The Bohr lift $\mathcal{B}f$ of a Dirichlet series $f(s) = \sum_{n=1}^{\infty}a_nn^{-s}$ lies in the infinite polydisk algebra $A(\mathbb{T}^{\infty})$, defined as the norm-closure of the algebra of analytic polynomials $\mathcal{P}_{\infty}$ in $C(\mathbb{T}^{\infty})$, the Banach space of continuous functions on $\mathbb{T}^{\infty}$ (see for instance \cite{Gamelin-PLMS-1986}). By applying Birkhoff-Oxtoby theorem \cite{Queffelec-TRM-2020} yields
		\begin{equation}\label{Dipon-vasu-p2-eqn-003}
			\norm{f}_{0}= \int_{\mathbb{T}^{\infty}}\log(1+|\mathcal{B}f|)\,dm_\infty = \norm{\mathcal{B}f}_{0}.
		\end{equation}
		 An observation follows from \cite{GuoZhou-AnnalFourier-2025} is that, if $f$ is a complex measurable functions on $\mathbb{T}^\infty$, then $\exp({\int_{\mathbb{T}^{\infty}}\log|f|\,dm_\infty})$ is finite if, and only if, $\exp({\int_{\mathbb{T}^{\infty}}\log(1+|f|)\,dm_\infty})$ is finite. It is well-known that, if $f\in L^r(\mathbb{T}^\infty)$ for some $0<r\le\infty$, then as $p\to 0^+$
		\begin{equation*}
			\norm{f}_p = \bigg({\int_{\mathbb{T}^{\infty}}|f|^p\, dm_\infty}\bigg)^{\frac{1}{p}}\to \exp\bigg({\int_{\mathbb{T}^{\infty}}\log|f|\,dm_\infty}\bigg).
		\end{equation*}
		Hence the integral $\int_{\mathbb{T}^{\infty}}\log(1+|f|)\,dm_\infty$ exists. Guo \textit{et al.} \cite{GuoZhou-AnnalFourier-2025} established via Bohr's correspondence that the Nevanlinna-Dirichlet class \( \mathcal{N} \) is isometrically isomorphic to the Nevanlinna class in infinitely many variables $N(\mathbb{D}^{\infty}_1)$. Nevertheless, one may introduce a broader class of Dirichlet series of Nevanlinna type using \textit{der m-te Abschnitt}(truncation up to $m$-th variable), which is more naturally adapted to the $\mathcal{H}^p$ framework (see for instance \cite[Section 4]{Brevig-AdvMath-2021}).

		For each $0<r<1$, the function $F_{[r]}$ over the infinite polytorus $\mathbb{T}^\infty$ is defined as	$F_{[r]}(\chi) = F(r\chi_1, r^2\chi_2,\cdots,r^n\chi_n,\cdots),$
		where $\chi\in\mathbb{T}^{\infty}.$ By analogy with Fatou's radial limit theorem \cite{Duren-H^p space-1970, Shapiro-Book-1992}, every non-zero function in $N(\mathbb{D}^\infty_1)$ admits a radial limit.
		
		\begin{thm}\cite{GuoZhou-AnnalFourier-2025}\label{Dipon-vasu-p2-thm-07}
			Let $F \in N(\mathbb{D}^\infty_1)$ be a non-zero function. Then for almost every $\chi \in \mathbb{T}^{\infty}$, the radial limit
			\[
			F^*(\chi) = \lim_{r \to 1^-} F_{[r]}(\chi)
			\]
			exists. Moreover, the boundary values satisfy $\log|F^*| \in L^1(\mathbb{T}^\infty)$.
		\end{thm}

		While the abscissa of uniform convergence $\sigma_u$ is invariant under taking vertical limits, the abscissa of convergence $\sigma_c$ generally fails to be invariant. In our context, for any $f \in \mathcal{N}_u$ with $\sigma_u(f) \le 0$, the convergence abscissa satisfies $\sigma_c(f_\chi) \le 0$ for almost every character $\chi \in \mathbb{T}^\infty$. By contrast, for $f$ in the Nevanlinna--Dirichlet class $\mathcal{N}$ (the completion of $\mathcal{N}_u$ under the metric $\Vert{}\cdot\Vert{}_0$ given in \eqref{Dipon-vasu-p2-eqn-002}), the vertical limit $f_\chi$ is not automatically meaningful, requiring one to first establish a suitable character action on the completion. Nevertheless, the following result is an analogue of the classical property in the Hardy space $\mathcal{H}^2$ of Dirichlet series, which plays a key role in establishing the theory of the class $\mathcal{N}_u$ (see, for instance, \cite[Section 2]{Brevig-JFA-2020}). For the convenience of the reader, we discuss this fact below. 
		\begin{lem}\label{Dipon-vasu-p2-lem-09}
			For a Dirichlet series $f$ in the class $\mathcal{N}_u$ satisfying $\sigma_u(f)\le 0,$ the generalized boundary value
			\begin{equation}\label{Dipon-vasu-p2-eqn-035}
				f^*(\chi) = \lim_{\sigma \to 0^+} f_\chi(\sigma)
			\end{equation}
			exists almost everywhere on $\mathbb{T}^\infty.$
		\end{lem}
		\begin{pf}
			Let $f(s) = \sum_{n \ge 1} a_n n^{-s}$ be a Dirichlet series in the class $\mathcal{N}_u$. For almost every completely multiplicative character $\chi \in \mathbb{T}^\infty$, we relate the vertical limit function
			\begin{equation}\label{Dipon-vasu-p2-eqn-033}
				f_\chi(it) = \sum_{n\ge 1} a_n \chi(n) n^{-it}
			\end{equation}
			to the Bohr lift $\mathcal{B}f$ of $f$. This relation follows from \cite[Section 4.2]{Hedenmalm-Duke-1997}. By Theorem \ref{Dipon-vasu-p2-thm-07}, the Bohr lift $\mathcal{B}f$ in $N(\mathbb{D}^\infty_1)$ admits a radial limit
			\[
			(\mathcal{B}f)^*(\chi) = \lim_{r \to 1^-} (\mathcal{B}f)_{[r]}(\chi)
			\]
			for almost every $\chi \in \mathbb{T}^\infty$. Specifically, for any character $\chi \in \mathbb{T}^\infty$ and $\tau \in \mathbb{R}$, consider the Kronecker flow
			\begin{equation}\label{Dipon-vasu-p2-eqn-034}
				T_\tau(\chi) = (2^{-i\tau}\chi_1, 3^{-i\tau}\chi_2, \dots, p_j^{-i\tau}\chi_j, \dots),
			\end{equation}
			where $p_j$ is the $j$-th prime number.
			Then the series \eqref{Dipon-vasu-p2-eqn-033} is well-defined and satisfies $$f_\chi(i\tau) = (\mathcal{B}f)^*(T_\tau(\chi)).$$
						Since the flow $T_\tau$ preserves Haar measure on $\mathbb{T}^\infty$, we obtain
			\begin{equation*}
			\int_{\mathbb{T}^{\infty}} \log(1 + |f_\chi(it)|) \,dm_\infty(\chi)	= \int_{\mathbb{T}^{\infty}} \log(1 + |(\mathcal{B}f)^*|) \, dm_\infty.
			\end{equation*}
			 Applying the Birkhoff--Khinchin ergodic theorem to the flow \eqref{Dipon-vasu-p2-eqn-034}, we obtain
			\begin{equation}\label{Dipon-vasu-p2-eqn-036}
				\int_{\mathbb{T}^{\infty}} \log(1 + |(\mathcal{B}f)^*(\chi)|) \, dm_\infty(\chi) = \lim_{T \to \infty} \frac{1}{2T} \int_{-T}^{T} \log(1 + |f_\chi(it)|) \, dt
			\end{equation}
			for almost every $\chi \in \mathbb{T}^\infty$. It is obvious that, for almost every $\chi \in \mathbb{T}^\infty$, the vertical limit function $f_\chi$ belongs to the conformally invariant Nevanlinna class $N(\mathbb{C}_0)$ associated with a suitable Cayley transform from $\mathbb{D}$ onto $\mathbb{C}_0$ (see the definition \ref{Dipon-vasu-p2-defn-01}).
			Consequently, the Dirichlet series $f_\chi$ converges for almost every $\chi \in \mathbb{T}^\infty$ and admits the requisite generalized boundary value \eqref{Dipon-vasu-p2-eqn-035}. This completes the proof.
		\end{pf}
		
		Furthermore, for any finite Borel measure $\mu$ on $\mathbb{R}$ by normalizing $\mu(\mathbb{R})=1$, Fubini's theorem yields
		\begin{align*}
			\int_{\mathbb{T}^{\infty}} \int_{\mathbb{R}} \log(1 + |f_\chi(it)|) \, d\mu(t) \, dm_\infty(\chi) &= \int_{\mathbb{R}} \int_{\mathbb{T}^{\infty}} \log(1 + |f_\chi(it)|) \, dm_\infty(\chi) \, d\mu(t) \\
			&= \|f\|_0.
		\end{align*}
		Consequently, it follows from \eqref{Dipon-vasu-p2-eqn-036} that the norm for a Dirichlet series $f \in \mathcal{N}_u$ takes the form
		\begin{equation}\label{Dipon-vasu-p2-eqn-032}
			\norm{f}_0 = \int_{\mathbb{T}^{\infty}} \log(1+|f^*(\chi)|)\,dm_\infty(\chi).
		\end{equation}
		Observe that this definition immediately guarantees the isometry $\Vert{}f\Vert{}_0 = \Vert{}f^*\Vert{}_0$.
\subsection{Carleson measures}\label{Dipon-vasu-p2-sec-2.4} The notion of Carleson measures arises naturally in the embedding problem for Hardy spaces of Dirichlet series $\mathcal{H}^p$: given $p \in [1,\infty)$, does there exist a constant $C > 0$ such that
\begin{equation}\label{Dipon-vasu-p2-eqn-041}
	\lim_{\sigma \to {\frac{1}{2}}^+} \int_{J} |f(\sigma+it)|^p \, dt \le C \norm{f}_{\mathcal{H}^p}^p
\end{equation}
holds for every bounded interval $J \subset \mathbb{R}$? Remarkably, the inequality \eqref{Dipon-vasu-p2-eqn-041} is known to be positive only for exponents $p\in 2\mathbb{N}$ and negative if $P<2$ as a corollary to a result of Harper \cite{Harper-ForumPi-2020}. In the theory of Carleson measures for classical Hardy spaces $H^p(\mathbb{D})$ and Hardy spaces of Dirichlet series $\mathcal{H}^p$, the notion of a Carleson square plays a central role. 

\begin{thm}\cite{Carleson-Annals-1962}
	Let $\mu$ be a nonnegative Borel measure on $\overline{\mathbb{C}_{1/2}}$ (or $\overline{\mathbb{D}}$) and let $H$ denote the space $\mathcal{H}^p$ (or $H^p(\mathbb{D})$). Then there exists an absolute constant $C > 0$ such that
	\begin{equation*}
		\norm{\mu}_{C,H} \le \sup_{Q} \frac{\mu(Q)}{\ell(Q)},
	\end{equation*}
	where the supremum is taken over all Carleson squares $Q$ in $\overline{\mathbb{C}_{1/2}}$ (or $\overline{\mathbb{D}}$).
\end{thm}
For a bounded interval $I = [\tau - h/2, \tau + h/2] \subset \mathbb{R}$ of side length $\ell(I) = h > 0$, the associated \textit{Carleson square} in $\mathbb{C}_0$ is defined by
\begin{equation}\label{Dipon-vasu-p2-eqn-042}
	Q_0(I) := \bigl\{ \sigma + it \in \mathbb{C}_0 : 0 < \sigma < h, \; t \in I \bigr\}.
\end{equation}
Thus, $Q_0(I)$ is a square of side length $h$ with its base aligned on the boundary line $i\mathbb{R}$. A positive Borel measure $\mu$ on $\mathbb{C}_0$ is said to satisfy the \textit{geometric Carleson condition} if there exists a constant $C > 0$ such that
\begin{equation}\label{Dipon-vasu-p2-eqn-043}
	\mu\bigl(Q_0(I)\bigr) \le C \ell(I)
\end{equation}
for all intervals $I \subset \mathbb{R}$. Furthermore, $\mu$ satisfies the \textit{vanishing geometric Carleson condition} if
\begin{equation}\label{Dipon-vasu-p2-eqn-044}
	\lim_{\delta \to 0^+} \sup_{\substack{I \subset \mathbb{R} \\ \ell(I) \le \delta}} \frac{\mu\bigl(Q_0(I)\bigr)}{\ell(I)} = 0.
\end{equation}
\section{The Nevanlinna class of Dirichlet series}\label{Dipon-vasu-p2-sec-3}

\subsection{Disconnectness of $\mathcal{N}$}	In this section, we demonstrate that \(\mathcal{N}\) is disconnected. The argument proceeds by constructing a non-trivial functional on \(N(\mathbb{D}^{\infty}_1)\) that is sub-additive, continuous, and non-Archimedean. By virtue of Bohr's correspondence, \(\mathcal{N}\) will exhibit the same properties. For $0 < r < 1$ and $M > 0$, Guo \textit{et al.} \cite{GuoZhou-AnnalFourier-2025} introduced the $\ell^1$-subdomains
\begin{equation}\label{Dipon-vasu-p2-eqn-012}
	V_{r,M} = \{ \zeta \in \ell^1 : |\zeta|_1 < M \text{ and } \sup_{n \ge 1}|\zeta_n| < r \}.
\end{equation}
As $r \to 1^-$ and $M \to \infty$, the sets $V_{r,M}$ expand to exhaust $\mathbb{D}^\infty_1$. Notably, it can be observed that exponential of a constant function in both the class \(\mathcal{N}\) and \(N(\mathbb{D}^{\infty}_1)\) lies outside the component of the origin. Although the argument builds on the work of Shapiro and Shields \cite{ShapiroShields-Amer-1975} via the embedding of the classical Nevanlinna class $N(\mathbb{D})$ into $\mathcal{N}$, we include short details below for completeness. 

\begin{thm}
	The Nevanlinna-Dirichlet class $\mathcal{N}$ is disconnected.
\end{thm}
\begin{pf}
	Let $F$ be a function on $\mathbb{D}^\infty_1$, which we regard as the Bohr lift $\mathcal{B} f$ of a Dirichlet series $f$. We define a non-trivial functional $\Gamma$ on $N(\mathbb{D}^\infty_1)$ by 
	\begin{equation}\label{Dipon-vasu-p2-eqn-026}
		\Gamma(F) = \lim_{r\to 1}\,  \,\log^+|F_{[r]}(\zeta)|,\,\,\,\mbox{for every}\,\,\,\zeta\in\mathbb{D}^\infty_1.
	\end{equation}
	We will demonstrate that $F$ and $G$ lie in different components of $N(\mathbb{D}^\infty_1)$ whenever $\Gamma(F)\ne\Gamma(G).$ It is evident that $\Gamma$ is sub-additive due to the elementary inequality 
	$$\log^+(a+b) \le \log^+a + \log^+b + \log 2,$$ for every $a, b\ge 0.$ From \cite[Lemma 2.10]{GuoZhou-AnnalFourier-2025}, we have for $\zeta\in\mathbb{D}^\infty_1$ and for every $F\in N(\mathbb{D}^\infty_1)$, we have the following inequality:
	\begin{equation}\label{Dipon-vasu-p2-eqn-027}
		\log(1+|F(\zeta)|)\le \norm{\textbf{P}_\zeta}_\infty\norm{F}_{N(\mathbb{D}^\infty_1)},
	\end{equation}
	where 
	\begin{equation}\label{Dipon-vasu-p2-eqn-037}
		\textbf{P}_\zeta(w) = \prod_{n=1}^{\infty}\frac{1-|\zeta_n|^2}{|\zeta_n-w_n|^2},\,\,\,\,\,(w\in\mathbb{T}^\infty)
	\end{equation}
	is the Poisson kernel at $\zeta\in\mathbb{T}^{\infty}$ and $n\in\mathbb{N}.$  In fact, the family $\{\mathbf{P}_\zeta\}_{\zeta \in V_{r,M}}$ is uniformly bounded in $L^\infty(\mathbb{T}^\infty)$ for each $\zeta \in V_{r,M}$. Consequently, there exists a constant $L > 0$ such that \eqref{Dipon-vasu-p2-eqn-027} reduces to
	\begin{equation*}
		\log(1+|F(\zeta)|) \le L \|F\|_0.
	\end{equation*}
	It is obvious from the equation \eqref{Dipon-vasu-p2-eqn-027} that
	\begin{equation*}
		\Gamma(F+G) = \max\{\Gamma(F), \Gamma(G)\},\,\,\,\mbox{if}\,\,\,\Gamma(F)\ne\Gamma(G).
	\end{equation*}
	Hence $\Gamma$ satisfies the non-Archimedean property. Further, the inequality 
	\begin{align*}
		\Gamma(F) &= \lim_{r\to 1}\,\,\log^+|F_{[r]}(\zeta)| \\ &\le \lim_{r\to 1}\log(1+|F_{[r]}(\zeta)|)\\ & \le L \norm{F}_{N(\mathbb{D}^\infty_1)}
	\end{align*}
	ensures that $\Gamma$ is continuous on $N(\mathbb{D}^\infty_1).$  It is clear that for the following function 
	\begin{equation*}
		F_C = \exp(C):=( e^{c_1},e^{c_2},\cdots),\,\,\,\,\,(C=(c_1,c_2,\cdots)\in\mathbb{T}^\infty),
	\end{equation*}
	 we have $\Gamma(F_C) = C.$ Therefore, the function \(F_C\)  lies in different components of \(N(\mathbb{D}^\infty_1)\), and none of these components coincide with the component of the origin. The remainder of the proof proceeds analogously to that of \cite[Theorem 2.1]{ShapiroShields-Amer-1975}. This concludes the proof.
\end{pf}
\begin{rem}
The above result can also be interpreted simply as follows. Define the projection-type operator $P \colon N(\mathbb{D}^\infty_1) \to N(\mathbb{D})$ by$$Pf(z) = f(z, 0, 0, \dots)$$for every $f \in N(\mathbb{D}^\infty_1)$. The map $P$ maps $N(\mathbb{D}^\infty_1)$ continuously onto $N(\mathbb{D})$. Since $N(\mathbb{D})$ is disconnected (see for instance \cite{ShapiroShields-Amer-1975}), it follows that $N(\mathbb{D}^\infty_1)$ is also disconnected. In particular, the class $\mathcal{N}_u$ is disconnected.
\end{rem}

\subsection{Littlewood-Paley type identity on the class $\mathcal{N}_u$}

 In the classical framework of \( N(\mathbb{D}) \), Choa and Kim \cite{ChoaKim-PAMS-1997} established a Littlewood-Paley type identity using Green's theorem on a \( C^2 \)- smooth boundary. However, for the Dirichlet series analogue of the Nevanlinna class \( \mathcal{N} \), proving such an identity is not so straightforward. The main aim for the proof of Theorem \ref{Dipon-vasu-p2-thm-01} is to establish an analogue of the Hardy–Stein type identity in the setting of the Nevanlinna class of Dirichlet series, as formulated in the following lemma.

\begin{lem}\label{Dipon-vasu-p2-lem-06}
	Let $f\in\mathcal{N}_u$ with $f \not\equiv 0$. Then the function $w\mapsto\mathcal{A}(f, w)$ is continuously differentiable on $(0,\infty)$, and 
	\begin{equation*}
		\frac{\partial}{\partial w}\mathcal{A}(f, w) = -\lim_{T\to\infty}\frac{1}{2T}\int_{w}^{\infty}\int_{-T}^{T}\frac{|f'(\sigma+it)|^2}{|f(\sigma+it)|(1+|f(\sigma+it)|)^2}\,dt\, d\sigma.
	\end{equation*}
	Moreover, the limit converges uniformly on $(w_0,\infty)$ for every fixed $w_0>0.$
\end{lem}

 \begin{proof}
		Fix $w_0 > 0$. Since $\sigma_u(f) \le 0$, the Dirichlet series representing $f$ and its termwise derivative converge uniformly on $\mathbb{C}_{w_0}$. In particular, $f$ and $f'$ are bounded on every half-plane $\mathbb{C}_{w_0}$. For $\varepsilon > 0$, define the smooth regularization
		\begin{equation*}
			u_\varepsilon(s) := \log\Bigl(1 + \sqrt{|f(s)|^2 + \varepsilon^2}\Bigr), \qquad s \in \mathbb{C}_0.
		\end{equation*}
		Then $u_\varepsilon \in C^\infty(\mathbb{C}_0)$ and satisfies $0 \le u_\varepsilon(s) - \log(1 + |f(s)|) \le \varepsilon$. Consequently, setting
		\begin{equation*}
			\mathcal{A}_\varepsilon(f, \sigma) := \lim_{T \to \infty} \frac{1}{2T} \int_{-T}^T u_\varepsilon(\sigma + it) \, dt,
		\end{equation*}
		we have
		\begin{equation*}
			\mathcal{A}_\varepsilon(f, \sigma) \longrightarrow \mathcal{A}(f, \sigma) \qquad \text{uniformly for } \sigma \ge w_0 \text{ as } \varepsilon \to 0^+.
		\end{equation*}
		
		We first establish the corresponding identity for $\mathcal{A}_\varepsilon$. Set $r_\varepsilon(s) := \sqrt{|f(s)|^2 + \varepsilon^2}$. A direct calculation utilizing $\Delta\bigl(\psi(|f|^2)\bigr) = 4 |f'|^2 \bigl( \psi'(|f|^2) + |f|^2 \psi''(|f|^2) \bigr)$ shows that
		\begin{equation*}
			\Delta u_\varepsilon(s) = |f'(s)|^2 \, \frac{r_\varepsilon(s)^2 + \varepsilon^2 \bigl(1 + 2r_\varepsilon(s)\bigr)}{r_\varepsilon(s)^3 \bigl(1 + r_\varepsilon(s)\bigr)^2} \ge 0,
		\end{equation*}
		where $\psi$ is a $C^2$-smooth function with supp $\psi'\cap(-\infty,0]$ is compact
		(see for instance \cite[Lemma 4.1]{BrevigKouro-TAMS-2025}).
		Moreover, as $\varepsilon \to 0^+$,
		\begin{equation}\label{Dipon-vasu-p2-eqn-052}
			\Delta u_\varepsilon(s) \longrightarrow q_f(s) := \frac{|f'(s)|^2}{|f(s)|\bigl(1 + |f(s)|\bigr)^2}
		\end{equation}
		locally in $L^1(\mathbb{C}_0)$. Away from the zero set $Z(f)$, the convergence is locally uniform. Near a zero $s_0 \in Z(f)$ of order $m \ge 1$, factoring $f(s) = (s - s_0)^m g(s)$ with $g(s_0) \neq 0$ yields $q_f(s) = O\bigl(|s - s_0|^{m-2}\bigr)$ as $s \to s_0$, which is locally integrable. Since $Z(f)$ consists of isolated points, local $L^1$ convergence in \eqref{Dipon-vasu-p2-eqn-052} follows.
		
		Now fix $w_0 < w < R$. For $T > 0$, applying Green's identity on the rectangle $R_{w, R, T} := \{ \sigma + it : w < \sigma < R, \ |t| < T \}$ gives
		\begin{align}
			\int_{-T}^T \frac{\partial u_\varepsilon}{\partial \sigma}(R + it) \, dt &- \int_{-T}^T \frac{\partial u_\varepsilon}{\partial \sigma}(w + it) \, dt \nonumber \\
			&+ \int_w^R \left[ \frac{\partial u_\varepsilon}{\partial t}(\sigma + iT) - \frac{\partial u_\varepsilon}{\partial t}(\sigma - iT) \right] d\sigma = \int_w^R \int_{-T}^T \Delta u_\varepsilon(\sigma + it) \, dt \, d\sigma. \label{Dipon-vasu-p2-eqn-053}
		\end{align}
		Since $\left| \frac{\partial u_\varepsilon}{\partial t}(s) \right| \le |f'(s)|$ uniformly in $\varepsilon > 0$ on $\mathbb{C}_w$, the horizontal side integrals in \eqref{Dipon-vasu-p2-eqn-053}, when divided by $2T$, vanish as $T \to \infty$. Uniform almost periodicity of $\partial_\sigma u_\varepsilon$ on $\mathbb{C}_w$ ensures the existence of its vertical mean, yielding
		\begin{equation*}
			\mathcal{A}_\varepsilon'(f, R) - \mathcal{A}_\varepsilon'(f, w) = \lim_{T \to \infty} \frac{1}{2T} \int_w^R \int_{-T}^T \Delta u_\varepsilon(\sigma + it) \, dt \, d\sigma,
		\end{equation*}
		where $\mathcal{A}_\varepsilon'(f, x)$ represents the partial derivative of $\mathcal{A}_\varepsilon(f, x)$ respect to $x$.
		Passing to the limit $\varepsilon \to 0^+$, then local $L^1$ convergence \eqref{Dipon-vasu-p2-eqn-052} and uniform bounds on $\mathbb{C}_w$ imply $\mathcal{A}_\varepsilon'(f, \sigma) \to \mathcal{A}'(f, \sigma)$ locally uniformly for $\sigma \ge w_0$. Thus,
		\begin{equation}\label{Dipon-vasu-p2-eqn-054}
			\mathcal{A}'(f, R) - \mathcal{A}'(f, w) = \int_w^R E_f(\sigma) \, d\sigma,
		\end{equation}
		where
		\begin{equation*}
			E_f(\sigma) := \lim_{T \to \infty} \frac{1}{2T} \int_{-T}^T \frac{|f'(\sigma + it)|^2}{|f(\sigma + it)|\bigl(1 + |f(\sigma + it)|\bigr)^2} \, dt \ge 0.
		\end{equation*}
		In particular, $\mathcal{A}(f, \cdot)$ is convex on $(0, \infty)$. Writing $f(s) = a_1 + \sum_{n \ge 2} a_n n^{-s}$, uniform convergence on $\mathbb{C}_\delta$ ($\delta > 0$) yields $f(\sigma + it) \to a_1$ and $f'(\sigma + it) \to 0$ uniformly in $t$ as $\sigma \to \infty$. Hence,
		\begin{equation*}
			\mathcal{A}(f, \sigma) \longrightarrow \log(1 + |a_1|) \quad \text{and} \quad \mathcal{A}'(f, \sigma) \longrightarrow 0 \qquad \text{as } \sigma \to \infty.
		\end{equation*}
		
		Taking $R \to \infty$ in \eqref{Dipon-vasu-p2-eqn-054} yields
		\begin{equation}\label{Dipon-vasu-p2-eqn-060}
			-\mathcal{A}'(f, w) = \int_w^\infty E_f(\sigma) \, d\sigma = \int_w^\infty \lim_{T \to \infty} \frac{1}{2T} \int_{-T}^T \frac{|f'(\sigma + it)|^2}{|f(\sigma + it)|\bigl(1 + |f(\sigma + it)|\bigr)^2} \, dt \, d\sigma.
		\end{equation}
		Finally, because $E_f \in L^1([w_0, \infty))$, the tail integrals converge uniformly for $w \ge w_0$. Combining uniform almost periodicity on $\mathbb{C}_{w_0}$ with the uniform convergence of $f$ and $f'$, we interchange the limit $T \to \infty$ and integration over $[w, \infty)$. This completes the proof.
	\end{proof}

Since $w\mapsto\mathcal{A}(f, w)$ is continuously differentiable from the Lemma \eqref{Dipon-vasu-p2-lem-06}, we
can write
\begin{equation}\label{Dipon-vasu-p2-eqn-013}
	\mathcal{A}(f, \sigma_1) - \mathcal{A}(f, \sigma_0) = \int_{\sigma_0}^{\sigma_1}\frac{\partial}{\partial w} \mathcal{A}(f,w)\,dw,
\end{equation}
for $\sigma_1>\sigma>\sigma_0>0.$ It is obvious that $f(s)\to a_1=f(+\infty)$ as $\real(s)\to\infty$. Further, we have $\mathcal{A}(f,\sigma_1)\to\log(1+|f(+\infty)|^2)$ as $\sigma_1\to\infty$, and $\mathcal{A}(f,\sigma_0)\to\norm{f}_{0}$ as $\sigma_0\to 0.$ Then  Theorem \ref{Dipon-vasu-p2-thm-01} will follow directly.


\begin{pf}[\bf{Proof of Theorem \ref{Dipon-vasu-p2-thm-01}}] 
	In this approach, we primarily utilize the estimate given by  \eqref{Dipon-vasu-p2-eqn-013} along with Hardy-Stein identity \ref{Dipon-vasu-p2-lem-06}. It is important to note that the limit in Lemma \ref{Dipon-vasu-p2-lem-06} is uniform, which allows us to take the limit outside of the integral. This leads to the following expression: 
	\begin{equation*}
		 \mathcal{A}(f, \sigma_0) = \mathcal{A}(f, \sigma_1) + \lim_{T\to\infty}\frac{1}{2T}\int_{\sigma_0}^{\sigma_1}\int_{w}^{\infty}\int_{-T}^{T}\frac{|f'(\sigma+it)|^2}{|f(\sigma+it)|(1+|f(\sigma+it)|)^2}\,dt\, d\sigma \, dw.
	\end{equation*}
	For a fixed $T > 0,$ we divide the integral over $\sigma$ into two parts. In the first part, we apply Tonelli’s theorem and change the order of integration to obtain
	\begin{align}\label{Dipon-vasu-p2-eqn-014}
		\int_{\sigma_0}^{\sigma_1}\int_{w}^{\sigma_1}\int_{-T}^{T}&\frac{|f'(\sigma+it)|^2}{|f(\sigma+it)|(1+|f(\sigma+it)|)^2}\,dt\, d\sigma \, dw\nonumber \\& \,\,\,\,\,\,\,\,\,\hspace{3cm}= \int_{\sigma_0}^{\sigma_1}\int_{-T}^{T}\frac{|f'(\sigma+it)|^2}{|f(\sigma+it)|(1+|f(\sigma+it)|)^2}(\sigma - \sigma_0)\,dt\, d\sigma.
	\end{align}
	For the second part, we get
	\begin{align}\label{Dipon-vasu-p2-eqn-015}
		\int_{\sigma_0}^{\sigma_1}\int_{\sigma_1}^{\infty}\int_{-T}^{T}&\frac{|f'(\sigma+it)|^2}{|f(\sigma+it)|(1+|f(\sigma+it)|)^2}\,dt\, d\sigma \, dw \nonumber\\& \,\,\,\,\,\,\,\,\,\hspace{3cm}= (\sigma_1 - \sigma_0) \int_{\sigma_1}^{\infty}\int_{-T}^{T}\frac{|f'(\sigma+it)|^2}{|f(\sigma+it)|(1+|f(\sigma+it)|)^2}\,dt\, d\sigma.
	\end{align}
By combining \eqref{Dipon-vasu-p2-eqn-014} and \eqref{Dipon-vasu-p2-eqn-015}, we can derive the following 
\begin{align*}
			\mathcal{A}(f, \sigma_0) &= \mathcal{A}(f, \sigma_1) + \lim_{T\to\infty}\frac{1}{2T}\int_{\sigma_0}^{\sigma_1}\int_{-T}^{T}\frac{|f'(\sigma+it)|^2}{|f(\sigma+it)|(1+|f(\sigma+it)|)^2}(\sigma - \sigma_0)\,dt\, d\sigma \\  & \,\,\,\,\,\,\,\,\,\hspace{7cm}- (\sigma - \sigma_0)\frac{\partial}{\partial\sigma_1}\mathcal{A}(f,\sigma_1).
\end{align*}
It is clear that the final term undergoes exponential decay as $\sigma_1 \to \infty$. Thus, we can effectively secure the desired result by initially allowing $\sigma_1$ to $\infty$, followed by permitting $\sigma_0$ to approach $0^+.$ This completes the proof.
\end{pf}

We now proceed to the proof of Theorem \ref{Dipon-vasu-p2-thm-04}, which is inspired by our approach used in the proof of Lemma \ref{Dipon-vasu-p2-lem-06}.

\begin{pf}[\bf{Proof of Theorem \ref{Dipon-vasu-p2-thm-04}}]
	Fix $T>0.$ Letting $\sigma_0\to 0^+$ and $\sigma_1\to\infty$, then denote
	\begin{align}\label{Dipon-vasu-p2-eqn-023}
		D_f(T) = \frac{1}{2T}\int_{-T}^{T}\log(1+|f(it)|)dt\,&-\log(1+|f(+\infty)|)\,\\&-\,\frac{1}{2T}\int_{-T}^{T}\int_{0}^{\infty}\frac{|f'(\sigma+it)|^2}{|f(\sigma+it)|(1+|f(\sigma+it)|)^2}\,\sigma \,d\sigma\, dt\nonumber.
	\end{align}
	We have to prove $D_f(T)\to 0$ as $T\to\infty$. For $\varepsilon>0$, define $f_\varepsilon(s) := f(s+\varepsilon).$ Since $\sigma_u(f)\le 0$, we have $f_\varepsilon(s)\in H^\infty(\mathbb{C}_0)$. Indeed $\norm{f_\varepsilon}_\infty<\infty.$ Then from \eqref{Dipon-vasu-p2-eqn-023}, it is enough to show 
		\begin{align}\label{Dipon-vasu-p2-eqn-024}
		\lim_{T \to \infty} \bigg|&\int_{-T}^{T}\log(1+|f(\varepsilon+it)|)dt\,-\log(1+|f(+\infty)|)\,\\&\,\,\hspace{3cm}-\,\int_{-T}^{T}\int_{0}^{\infty}\frac{|f'(\sigma+\varepsilon+it)|^2}{|f(\sigma+\varepsilon+it)|(1+|f(\sigma+\varepsilon+it)|)^2}\,(\rho-\varepsilon) \,d\sigma\, dt\nonumber\bigg|=0.
	\end{align}
	 Denote
	\begin{equation*}
			q_f(\rho,t) = \frac{|f'(\rho+it)|}{|f(\rho+it)|(1+|f(\rho+it)|)^2}.
	\end{equation*}
	Also from the Hardy-Stein identity \eqref{Dipon-vasu-p2-eqn-060}, we have
	\begin{equation}\label{Dipon-vasu-p2-eqn-025}
		\int_{0}^{\infty}\rho E_f(\rho)d\rho = \norm{f}_0 - \log(1+|f(+\infty)|)<\infty.
	\end{equation}
	Consequently, only the remaining part is to prove 
	\begin{equation}\label{Dipon-vasu-p2-eqn-061}
		\lim_{\epsilon\to 0^+}\lim_{T \to \infty}\frac{1}{2T}\int_{-T}^{T}\int_{0}^{\infty}\abs{\rho q_f(\rho,t)-(\rho - \varepsilon)_+ q_f(\rho,t)}d\rho\,dt=0,
	\end{equation}
		where
	\begin{equation*}
		(\rho - \varepsilon)_+ = \begin{cases} 0, & 0 < \rho < \varepsilon, \\ \rho - \varepsilon, & \rho \ge \varepsilon. \end{cases}
	\end{equation*}
	Then the integrand in \eqref{Dipon-vasu-p2-eqn-061} is $\rho q_f(\rho,t)$, when $0 < \rho < \varepsilon$ and $\varepsilon q_f(\rho,t)$, when $\rho \ge \varepsilon.$ 
	Hence \eqref{Dipon-vasu-p2-eqn-061} is equivalent to
	\begin{equation*}
		\lim_{\varepsilon \to 0^+} \left[ \int_{0}^{\varepsilon} \rho E_f(\rho) \, d\rho + \varepsilon \int_{\varepsilon}^{\infty} E_f(\rho) \, d\rho \right] = 0,
	\end{equation*}
	The equation \eqref{Dipon-vasu-p2-eqn-061} is obvious to establish from \eqref{Dipon-vasu-p2-eqn-025}. Indeed,
	$$\int_{0}^{\varepsilon} \rho E_f(\rho) \, d\rho \to 0,$$
	and
	$$\varepsilon \int_{\varepsilon}^{\infty} E_f(\rho) \, d\rho \le \int_{\varepsilon}^{\infty} \rho E_f(\rho) \, d\rho \to 0.$$
	Moreover $\sigma E_f(\sigma)\in L^1(0,\infty)$, the right-hand side
	tends to zero as $\varepsilon\to0^+$.
 This completes the proof.
\end{pf}

\subsubsection{Potential-theoretic approach to prove Theorem~\ref{Dipon-vasu-p2-thm-01}}

We now present an alternative proof of Theorem~\ref{Dipon-vasu-p2-thm-01} via a potential-theoretic approach, applying the Poisson--Jensen formula \cite[Theorem~4.5.1]{Ransford-Potential-1995} to a suitable subharmonic function.

\begin{thm}[\cite{Ransford-Potential-1995}, Poisson-Jensen Formula]\label{Dipon-vasu-p2-thm-015}
Let $D$ be a bounded regular domain in the complex plane $\mathbb{C}$, and let $u$ be a function subharmonic on a neighbourhood of $\overline{D}$, with $u\neq -\infty$ on $D$. Then 
\begin{equation*}
u(z) = \int_{\partial D}u(\zeta)\,dw_D(z,\zeta) - \frac{1}{2\pi}\int_{D}g_D(z,w)\, \Delta u(w), 
\end{equation*}
where $z, w\in D$ and $\zeta\in\partial D$. 
\end{thm}
In Theorem \ref{Dipon-vasu-p2-thm-015}, $w_D(z,\zeta)$ denotes the harmonic measure while $g_D(z,w)$ represents Green's function associated with a proper subdomain $D$ of the Riemann sphere $\mathbb{C}\cup\{\pm\infty\}$. Since the half-plane domains are unbounded, additional information about the zero sets of functions in $\mathcal{H}^\infty$ is required (see for instance \cite{BrevigKouro-TAMS-2025}). This information will be derived from almost periodicity, drawing on classical results of Bohr and Jessen \cite{Bohr-ActaMath-1930}. Since the Littlewood--Paley identity can be extended to general domains using Green functions, we first establish a representation formula for an arbitrary rectangle in the half-plane. This potential-theoretic approach yields our desired Littlewood--Paley type identity (see Theorem~\ref{Dipon-vasu-p2-thm-01}) for functions in the Nevanlinna class of Dirichlet series $\mathcal{N}_u$.

\begin{lem}\label{Dipon-vasu-p2-lem-11}
	Let $S_{\sigma_0, R} := \{ s = \sigma + it \in \mathbb{C} : \sigma_0 < \sigma < R \}$ for $0 < \sigma_0 < R$, and let $G_{\sigma_0, R}$ denote the Green function of $S_{\sigma_0, R}$, normalized such that
	\begin{equation*}
		-\Delta_z G_{\sigma_0, R}(z, w) = 2\pi \delta_w.
	\end{equation*}
	Then, for all $\sigma, \rho \in (\sigma_0, R)$,
	\begin{equation}\label{Dipon-vasu-p2-eqn-008}
		\int_{\mathbb{R}} G_{\sigma_0, R}(\sigma + it, \rho) \, dt = 2\pi \, \frac{(\min\{\sigma, \rho\} - \sigma_0)(R - \max\{\sigma, \rho\})}{R - \sigma_0}.
	\end{equation}
	In particular, taking the limit as $R \to \infty$ yields
	\begin{equation}\label{Dipon-vasu-p2-eqn-009}
		\lim_{R \to \infty} \int_{\mathbb{R}} G_{\sigma_0, R}(\sigma + it, \rho) \, dt = 2\pi (\min\{\sigma, \rho\} - \sigma_0).
	\end{equation}
\end{lem}

\begin{proof}
	Define $A(\sigma, \rho) := \int_{\mathbb{R}} G_{\sigma_0, R}(\sigma + it, \rho) \, dt$. By the vertical translation invariance of the strip $S_{\sigma_0, R}$, the integral $A(\sigma, \rho)$ is independent of $\operatorname{Im}(\rho)$. Integrating the distributional equation $$-\Delta_z G_{\sigma_0, R}(z, \rho) = 2\pi \delta_\rho(z)$$ with respect to $t \in \mathbb{R}$, the term involving $\partial^2/\partial t^2$ vanishes, leaving
	\begin{equation*}
		-\frac{\partial^2}{\partial \sigma^2} A(\sigma, \rho) = 2\pi \delta_\rho(\sigma).
	\end{equation*}
	Furthermore, the boundary conditions for $G_{\sigma_0, R}$ enforce $A(\sigma_0, \rho) = A(R, \rho) = 0$. Consequently, $(2\pi)^{-1} A(\sigma, \rho)$ is precisely the one-dimensional Green function for the Laplacian on the interval $(\sigma_0, R)$. Therefore,
	\begin{equation*}
		\frac{A(\sigma, \rho)}{2\pi} = \frac{(\min\{\sigma, \rho\} - \sigma_0)(R - \max\{\sigma, \rho\})}{R - \sigma_0},
	\end{equation*}
	establishing \eqref{Dipon-vasu-p2-eqn-008}. Taking $R \to \infty$ directly yields \eqref{Dipon-vasu-p2-eqn-009}. This completes the proof.
\end{proof}

\begin{thm}\label{Dipon-vasu-p2-thm-016}
	Let $f \in \mathcal{N}_u$ with $f \not\equiv 0$, and define $U_f(s) := \log(1 + |f(s)|)$ for $s \in \mathbb{C}_0$. Then, in the sense of distributions on $\mathbb{C}_0$, we have
	\begin{equation*}
		\Delta U_f = \frac{|f'|^2}{|f|(1+|f|)^2}.
	\end{equation*}
	Moreover, the $d_0$-norm of $f$ satisfies the Littlewood--Paley identity
	\begin{align}
		\|f\|_0 ={}& \log(1 + |f(+\infty)|) \nonumber \\
		&+ \limsup_{\sigma_0 \to 0^+} \lim_{T \to \infty} \frac{1}{2T} \int_{\sigma_0}^\infty \int_{-T}^T \frac{|f'(\sigma+it)|^2}{|f(\sigma+it)|\bigl(1+|f(\sigma+it)|\bigr)^2} (\sigma - \sigma_0) \, dt \, d\sigma. \label{Dipon-vasu-p2-eqn-010}
	\end{align}
\end{thm}
\begin{pf}
	Here $U_f$ is subharmonic on $\mathbb{C}_0$. Outside the zero set $Z(f)$, $$\Delta U_f = |f'|^2 / \bigl(|f|(1+|f|)^2\bigr).$$ Near a zero $s_0$ of multiplicity $m \ge 1$, $|f'|^2 / \bigl(|f|(1+|f|)^2\bigr) = O(|s - s_0|^{m-2})$, which is locally integrable. Hence $\Delta U_f$ holds distributionally on $\mathbb{C}_0$ with Riesz measure
	\begin{equation}\label{eq:riesz-measure-u}
		d\mu_u(s) = \frac{|f'(s)|^2}{|f(s)|\bigl(1 + |f(s)|\bigr)^2} \, dA(s),
	\end{equation}
	possessing no point-masses at $Z(f)$ (see for instance \cite[Theorem 3.7.8]{Ransford-Potential-1995}).
	For $0 < \sigma_0 < R$ and $T > 0$, let $D_{R,T} = \{ \sigma + it : \sigma_0 < \sigma < R, \ |t| < T \}$. Denote 
	\begin{equation*}
		H_{R,T} := \int_{\partial D_{R,T}}U_f(\zeta)\,dw_{D_{R,T}}(z,\zeta),
	\end{equation*}
	which is harmonic in $D_{R,T}$ and $H_{R,T} = U_f$ on $\partial D_{R,T}$.
	Averaging the Green representation formula for $U_f$ on $D_{R,T}$ vertically yields
	\begin{equation}\label{eq:green-rep-u}
		U_{R,T}(\sigma) = H_{R,T}(\sigma) - \frac{1}{2\pi} \int_{D_{R,T}} \overline{G}_{R,T}(\sigma, w) \, \Delta U_f(w) \, dA(w),
	\end{equation}
	where $U_{R,T}(\sigma) = \frac{1}{2T} \int_{-T}^T U_f(\sigma + it) \, dt$ and $\overline{G}_{R,T}$ is the vertically averaged Green function of $D_{R,T}$. Moreover, averaging the Green's function in the vertical direction gives
	\begin{equation*}
		G_{\sigma_0, R}(\sigma, \rho) = \int_{\mathbb{R}} G_{\sigma_0, R}(\sigma + it, \rho) \, dt.
	\end{equation*}
	Then by Lemma \ref{Dipon-vasu-p2-lem-11}, we obtain as $R\to\infty$
	\begin{equation*}
		G_{\sigma_0, R}(\sigma, \rho) \rightarrow 2\pi (\min\{\sigma, \rho\} - \sigma_0).
	\end{equation*} 
	Now the Poisson--Jensen formula, after vertical averaging and passage $T \to \infty$, gives
	\begin{equation}\label{Dipon-vasu-p2-eqn-020}
		U_f(\sigma) = \frac{R-\sigma}{R-\sigma_0} U_f(\sigma_0) + \frac{\sigma-\sigma_0}{R-\sigma_0} U_f(R) - \int_{\sigma_0}^R K_{\sigma_0, R}(\sigma, \rho)\, V(\rho) \, d\rho,
	\end{equation}
	where
	\begin{equation*}
		V(\rho) = \lim_{T \to \infty} \frac{1}{2T} \int_{-T}^T \frac{|f'(\rho + it)|^2}{|f(\rho + it)|(1 + |f(\rho + it)|)^2} \, dt
	\end{equation*}
	and
	\begin{equation*}
		K_{\sigma_0, R}(\sigma, \rho) = \frac{(\min\{\sigma, \rho\} - \sigma_0)(R - \max\{\sigma, \rho\})}{R - \sigma_0}.
	\end{equation*}
	Letting $R \to \infty$, using $U_f(R) \longrightarrow \log(1 + |f(+\infty)|),$ we obtain from \eqref{Dipon-vasu-p2-eqn-020}
	\begin{equation}
		U_f(\sigma) = U_f(\sigma_0) - \int_{\sigma_0}^\infty (\min\{\sigma, \rho\} - \sigma_0) V(\rho) \, d\rho.
	\end{equation}
	More conveniently, letting $\sigma\to\infty$ and substituting the definition of $V$, we have 
		\begin{align}
		U_f(\sigma_0) ={}& \log(1 + |f(+\infty)|) \nonumber \\
		&+  \lim_{T \to \infty} \frac{1}{2T} \int_{\sigma_0}^\infty \int_{-T}^T \frac{|f'(\sigma+it)|^2}{|f(\sigma+it)|\bigl(1+|f(\sigma+it)|\bigr)^2} (\sigma - \sigma_0) \, dt \, d\sigma. 
	\end{align}
	Finally from the definition of $\mathcal{N}_u$, 
	\begin{equation*}
		\norm{f}_0 = \limsup_{\sigma_0 \to 0^+}U_f(\sigma_0)
	\end{equation*}
	and \eqref{Dipon-vasu-p2-eqn-010} follows. This concludes the proof.
	
\end{pf}

\subsubsection{Littlewood-Paley type formula \eqref{Dipon-vasu-p2-eqn-004} in terms of the vertical limit functions}
To establish our new Littlewood-Paley type formula  for the class $\mathcal{N}_u$ in terms of the vertical limit functions $f_\chi$ ($\chi \in \mathbb{T}^\infty$), we require the following lemma:

 \begin{lem}\cite{BrevigKouro-TAMS-2025}\label{Dipon-vasu-p2-lem-04}
	Let $\mu$ be a finite Borel measure on $\mathbb{R}$ and $\mathcal{F}\in L^1(\mathbb{T}^\infty).$ Then 
	\begin{equation*}
		\norm{\mathcal{F}}_{L^1(\mathbb{T}^\infty)}\mu(\mathbb{R}) = \int_{\mathbb{T}^{\infty}}\int_{\mathbb{R}}\mathcal{F} (\mathcal{T}_t\chi)\,d\mu(t)\,dm_\infty(\chi).
	\end{equation*}
\end{lem}
 
 Here, we introduce the \textit{conformally invariant Nevanlinna class} on the  half-plane $\mathbb{C}_0$, denoted by $N(\mathbb{C}_0)$. 
 \begin{defn}\label{Dipon-vasu-p2-defn-01}
 	A holomorphic function $f$ on $\mathbb{C}_0$ belongs to $N(\mathbb{C}_0)$ if $f \circ \Psi \in N(\mathbb{D})$, where$$\Psi(z) = \frac{1+z}{1-z}$$is the Cayley transformation mapping the open unit disk $\mathbb{D}$ onto $\mathbb{C}_0$.
 \end{defn}
  Furthermore, let $\mu$ be the probability measure on $\mathbb{R}$ defined by$$d\mu(t) = \frac{1}{\pi(1+t^2)}\,dt,$$which is the pushforward of the normalized Lebesgue (Haar) measure on the unit circle $\mathbb{T}$ under the boundary map of $\Psi$. Moreover we can write
 \begin{equation}\label{Dipon-vasu-p2-eqn-029}
 	\int_{\mathbb{R}}\log(1+|f(it)|)\,d\mu(t) = \frac{1}{2\pi}\int_{-\pi}^{\pi}\log(1+|f\circ\Psi(e^{i\theta})|)\,d\theta.
 \end{equation}
 More precisely, for a planar domain $D \subset \mathbb{C}$, the class $N(D)$ is conformally invariant in the following sense: if $\Psi \colon D \to D'$ is a conformal equivalence, and $p' = \Psi(p)$ is the basepoint used to define the metric on $N(D')$, then the composition operator $C_\Psi \colon N(D') \to N(D)$, defined by $f \mapsto f \circ \Psi$ is an isometric isomorphism.

\begin{lem}\label{Dipon-vasu-p2-lem-05}
	Let $\mu$ be a finite Borel measure on $\mathbb{R}$. Then 
	\begin{equation*}
		\norm{f}_0\mu(\mathbb{R}) = \int_{\mathbb{T}^{\infty}}\int_{\mathbb{R}}\log(1+|f_\chi(it)|)\,d\mu(t)\,dm_\infty(\chi).
	\end{equation*}
\end{lem}

\begin{pf}
 Vertical translations generate the flow of an ergodic rotation on the polytorus $\mathbb{T}^\infty$. For $t \in \mathbb{R}$, we define the flow$$\mathcal{T}_t\chi = (2^{it}\chi_1, \dots, p_j^{it}\chi_j, \dots), \quad \chi \in \mathbb{T}^\infty,$$ where $p_j$ is the $j$-th prime number. Setting $\mathcal{F}(\chi) = \log(1+\vert{}f_\chi(0)\vert{})$ as in Lemma \ref{Dipon-vasu-p2-lem-04}, the Birkhoff--Khinchin ergodic theorem \cite{Queffelec-TRM-2020} implies that for almost all $\chi$,
 \begin{equation*}\lim_{T\to\infty}\frac{1}{2T}\int_{-T}^{T}\log(1+|f_\chi(it)|)\,dt = \int_{\mathbb{T}^{\infty}}\log(1+|\mathcal{B}f|)\,dm_\infty = \norm{f}_0.\end{equation*}
 The result then follows by Fubini's theorem.
\end{pf}

\begin{thm}\label{Dipon-vasu-p2-thm-08}
	Let $\mu$ be a probability measure on $\mathbb{R}$ and  $f(s)=\sum_{n\ge 1}a_nn^{-s}$ be a Dirichlet series in $\mathcal{N}_u$. Then, for $f\not\equiv 0$, we have
	\begin{equation*}
			\norm{f}_{0}\asymp \log(1+ |f_\chi(+\infty)|) + 4\int_{\mathbb{T}^{\infty}}\int_{0}^{\infty}\int_{\mathbb{R}}\frac{|f_\chi'(\sigma+it)|^2}{|f_\chi(\sigma+it)|(1+|f_\chi(\sigma+it)|)^2}\,\sigma \,d\mu(t\,) \,d\sigma\, dm_\infty(\chi).
	\end{equation*}
\end{thm}
\begin{pf}
	By the virtue of the Littlewood--Paley identity for the classical Nevanlinna class $N(\mathbb{D})$ by Choa and Kim \cite{ChoaKim-PAMS-1997}, we get an analogues identity
	\begin{equation*}
		\norm{g}_{N(\mathbb{D})} = \log(1+|g(0)|) + 2\iint_{\mathbb{D}}\frac{|g'(z)|^2}{|g(z)|(1+|g(z)|)^2}\log(1/|z|)\,dA(z), \quad g \in N(\mathbb{D}),
	\end{equation*}
	where $$\Vert{}g\Vert{}_{N(\mathbb{D})} := \lim_{r\to 1^-}\int_{0}^{2\pi} \log(1+\vert{}g(re^{i\theta})\vert{})\,\frac{d\theta}{2\pi}$$ is the standard Nevanlinna quasi-norm and $dA(z)$ denotes Lebesgue area measure on $\mathbb{D}$. Let $f$ be a Dirichlet polynomial. For a fixed parameter $\xi > 0$, the Cayley transformation$$\Psi_\xi(z) = \xi\frac{1+z}{1-z}$$maps $\mathbb{D}$ bijections onto $\mathbb{C}_0$, and its inverse is given by$$\Psi_\xi^{-1}(s) = \frac{s-\xi}{s+\xi}.$$
By Lemma \ref{Dipon-vasu-p2-lem-05}, we have
\begin{equation}\label{Dipon-vasu-p2-eqn-030}
	\norm{f}_0 = \int_{\mathbb{T}^{\infty}}\left( \int_{\mathbb{R}}\log(1+|f_\chi(it)|)\,\frac{\xi}{\pi(\xi^2+t^2)}\,dt \right) dm_\infty(\chi).
\end{equation}
For a fixed character $\chi \in \mathbb{T}^\infty$, \eqref{Dipon-vasu-p2-eqn-029} yields
	\begin{align*}
		\int_{\mathbb{R}}&\log(1+|f_\chi(it)|)\,\frac{\xi}{\pi(\xi^2+t^2)}dt \\ &= \norm{f_\chi\circ\Psi_\xi}_{N(\mathbb{D})} \\ &\asymp \log(1+|f_\chi\circ\Psi_\xi(0)|) + 4\iint_{\mathbb{D}}(1-|z|^2)\frac{|f'_\chi\circ\Psi_\xi(z)|^2}{|f_\chi\circ\Psi_\xi(z)|(1+|f_\chi\circ\Psi_\xi(z)|)^2}|\Psi'_\xi(z)|^2\,dA(z).
	\end{align*}
	Applying the change of variables $s=\sigma+it=\Psi_\xi(z)$, we obtain
	\begin{align*}
		\int_{\mathbb{R}}&\log(1+|f_\chi(it)|)\,\frac{\xi}{\pi(\xi^2+t^2)}dt \\ &\asymp \log(1+|f_\chi(\xi)|) + 4\int_{0}^{\infty}\int_{\mathbb{R}}\left(1-\frac{|s-\xi|^2}{|s+\xi|^2}\right)\frac{|f'_\chi(s)|^2}{|f_\chi(s)|(1+|f_\chi(s)|)^2}\,dt\,d\sigma \\ & \asymp \log(1+|f_\chi(\xi)|) + 4\int_{0}^{\infty}\int_{\mathbb{R}}\frac{\sigma\xi}{((\sigma+\xi)^2+t^2)}\frac{|f'_\chi(s)|^2}{|f_\chi(s)|(1+|f_\chi(s)|)^2}\,dt\,d\sigma.
	\end{align*}
	Substituting \eqref{Dipon-vasu-p2-eqn-030} and integrating over the polytorus $\mathbb{T}^\infty$, we obtain
	\begin{align*}
		\norm{f}_0 & \asymp \int_{\mathbb{T}^{\infty}}\log(1+|f_\chi(\xi)|)\,dm_\infty(\chi) \\ & + 4\int_{0}^{\infty}\frac{\sigma\xi}{\sigma+\xi}\int_{\mathbb{T}^{\infty}}\int_{\mathbb{R}}\frac{|f'_\chi(s)|^2}{|f_\chi(s)|(1+|f_\chi(s)|)^2} \frac{\sigma+\xi}{\pi(\sigma+\xi)^2+t^2}\,dt\,dm_\infty(\chi)\,d\sigma.
	\end{align*}
	By virtue of Lemma \ref{Dipon-vasu-p2-lem-04} applied to the probability measure $d\mu(t) = \frac{\sigma+\xi}{\pi(\sigma+\xi)^2+t^2}\,dt$, together with Fubini's theorem, we obtain
    \begin{align*}
    	\norm{f}_0 & \asymp \int_{\mathbb{T}^{\infty}}\log(1+|f_\chi(\xi)|)\,dm_\infty(\chi) \\ & + 4\int_{0}^{\infty}\frac{\sigma\xi}{\sigma+\xi}\int_{\mathbb{T}^{\infty}}\int_{\mathbb{R}}\frac{|f'_\chi(s)|^2}{|f_\chi(s)|(1+|f_\chi(s)|)^2}\,d\mu(t)\,dm_\infty(\chi)\,d\sigma.
    \end{align*}
   Applying Fubini's theorem to interchange the integration with respect to $\sigma$ and the Haar measure $m_\infty$, and taking the limit as $\xi \to \infty$, we arrive at the desired result.
\end{pf}

\begin{rem}
Furthermore, it follows from the foundational work of Brevig and Perfekt \cite[Lemma 3.4]{Brevig-AdvMath-2021} that
\begin{equation*}
	\norm{f_\chi}_0 = \norm{f}_0 \quad\text{and}\quad \log(1+|f_\chi(+\infty)|) = \log(1+|f(+\infty)|)
\end{equation*}
for every character $\chi \in \mathbb{T}^{\infty}$. Applying an argument analogous to that of \cite[Theorem 4.2]{BrevigKouro-TAMS-2025}, we obtain
\begin{equation}\label{Dipon-vasu-p2-eqn-031}
	\norm{f}_{0}= \log(1+ |f(+\infty)|) + \lim_{T\to\infty}\frac{2}{T}\int_{-T}^{T}\int_{0}^{\infty}\frac{|f_\chi'(\sigma+it)|^2}{|f_\chi(\sigma+it)|(1+|f_\chi(\sigma+it)|)^2}\,\sigma \,d\sigma\, dt.
\end{equation}
Alternatively, formula \eqref{Dipon-vasu-p2-eqn-031} can be established by applying the Birkhoff--Khinchin ergodic theorem (see for instance \cite{Bayart-monatsh-2002,Queffelec-TRM-2020}).

\end{rem}

\section{Compact composition operators on $\mathcal{N}_u$}\label{Dipon-vasu-p2-sec-4}

The theory of composition operators is a rich and interdisciplinary area of research, exhibiting deep connections with complex analysis, linear dynamics, complex geometry, and functional analysis. The composition operator on the Hardy space of Dirichlet series $\mathcal{H}^p$, for $0<p\le\infty$, has been studied by several authors (see for instance \cite{Bayart-monatsh-2002,Bayart-IJ-2003,Gordon-Michigan-1999,Hedenmalm-Duke-1997,Athanos-JLMS-2023}). By Bohr's classical point of view \cite{Bohr-Acta-1913}, $\mathcal{H}^p$ is identified with the Hardy space $\mathbb{H}^p(\mathbb{T}^\infty)$ on the infinite-dimensional torus. 
According to a theorem of Cole and Gamelin \cite{Gamelin-PLMS-1986}, point evaluation at $z$ defines a bounded linear functional on $\mathbb{H}^p(\mathbb{T}^{\infty})$ if and only if $z \in \mathbb{D}_2^\infty = \mathbb{D}^\infty \cap \ell^2$. Guo \textit{et al.} \cite[Theorem 4.1]{GuoZhou-AnnalFourier-2025} proved that for a Dirichlet series $f$ in $\mathcal{N}_u$, it's corresponding Bohr lift $\mathcal{B}f$ converges in $\mathbb{D}^\infty_1.$ Moreover $\norm{\mathcal{B}f}_0=\norm{f}_0.$ Guo \textit{et al.} \cite{GuoZhou-AnnalFourier-2025} showed that for any $\zeta \in \mathbb{D}^\infty_1$, the Bohr lift $\mathcal{B}f$ of a Dirichlet series $f$ satisfies the inequality
\begin{equation}\label{Dipon-vasu-p2-eqn-038}
	\log(1+|\mathcal{B}f(\zeta)|) \le \|\mathbf{P}_\zeta\|_\infty \|\mathcal{B}f\|_0,
\end{equation}
where $\mathbf{P}_\zeta$ is defined in \eqref{Dipon-vasu-p2-eqn-037}. In fact, the family $\{\mathbf{P}_\zeta\}_{\zeta \in V_{r,M}}$ is uniformly bounded in $L^\infty(\mathbb{T}^\infty)$ for each $\zeta \in V_{r,M}$, where $V_{r,M}$ is a $\ell^1$ subdomain defined in \eqref{Dipon-vasu-p2-eqn-012}. Consequently, there exists a constant $K > 0$ such that \eqref{Dipon-vasu-p2-eqn-038} reduces to
\begin{equation}\label{Dipon-vasu-p2-eqn-039}
	\log(1+|\mathcal{B}f(\zeta)|) \le K \|\mathcal{B}f\|_0.
\end{equation}
For $s \in \mathbb{C}_1$, setting $\zeta = (2^{-s}, 3^{-s}, \dots) \in \mathbb{D}^\infty_1$ ensures that $\mathcal{B}f \in N(\mathbb{D}^\infty_1)$. Applying \eqref{Dipon-vasu-p2-eqn-039} then yields
\begin{equation*}
	\log(1+|f(s)|) \le K \|f\|_0,
\end{equation*}
which proves the boundedness of the point evaluation functional at $s \in \mathbb{C}_1$ for Dirichlet series in the class $\mathcal{N}_u$.

Motivated by the seminal work of Gordon and Hedenmalm \cite{Gordon-Michigan-1999}, we obtain the following analogous result, whose proof is omitted as it follows directly without additional effort. Throughout this section, $\mathcal{D}$ denotes the space of functions on a half-plane that can be represented by a convergent Dirichlet series.

   \begin{thm}\label{Dipon-vasu-p2-thm-02}
		 A mapping $\Phi : \mathbb{C}_1 \to \mathbb{C}_1$ induces a bounded composition operator $C_\Phi : \mathcal{N}_u \to \mathcal{D}$ if and only if $\Phi(s) = c_0 s + \phi(s),$
		for some non-negative integer $c_0 \in \mathbb{N}\cup\{0\}$ and some $\phi \in \mathcal{D}$.
	\end{thm}

Moreover, Theorem \ref{Dipon-vasu-p2-thm-02} holds for a symbol $\Phi : \mathbb{C}_\theta \to \mathbb{C}_1$, for any $\theta\in\mathbb{R}.$	
For a Dirichlet series $f(s) = \sum_{n=1}^\infty a_n n^{-s}$, Helson \cite{Helson-Arkiv-1970} proved that the corresponding vertical limit function $f_\chi$ converges in the half-plane $\mathbb{C}_0$ for almost every character $\chi \in \mathbb{T}^\infty$. This convergence property extends to the Nevanlinna class $\mathcal{N}_u$ with $\sigma_u(f) \le 0$ in a similar way. Furthermore, in Lemma~\ref{Dipon-vasu-p2-lem-09}, we establish that $f_\chi$ admits generalized boundary values for almost every $\chi \in \mathbb{T}^\infty$, a fact that plays a key role in our subsequent results.

\begin{prop}\label{Dipon-vasu-p2-prop-01}
	If a holomorphic function $\Phi: \mathbb{C}_1 \to \mathbb{C}_1$ with char$(\Phi)=0$ induces a bounded composition operator $C_\Phi: \mathcal{N}_u \to \mathcal{N}_u$ on functions with $\sigma_u(f) \le 0$, then for almost every character $\chi \in \mathbb{T}^\infty$, the vertical limit function $\Phi_\chi$ admits an analytic continuation to the open right half-plane $\mathbb{C}_0$.
\end{prop}

\begin{proof}
	We focus on the primary case of interest where the symbol $\Phi$ with char$(\Phi)=0$. For each $n \in \mathbb{N}$, the Dirichlet monomial $f_n(s) = n^{-s}$ belongs to $\mathcal{N}_u$. By assumption, $C_\Phi(n^{-s}) = n^{-\Phi(s)} \in \mathcal{N}_u$. Consequently, the vertical limit function $(n^{-\Phi})_\chi$ is holomorphic in $\mathbb{C}_0$ for almost every character $\chi \in \mathbb{T}^\infty$. The identity
	\begin{equation*}
		n^{-\Phi_\chi(s)} = \chi(n)^{-c_0} (n^{-\Phi})_\chi(s)
	\end{equation*}
	provides an analytic continuation of $n^{-\Phi_\chi}$ to $\mathbb{C}_0$ for almost every character $\chi$ in the half-plane of uniform convergence of $\phi$. Indeed, for each $n \in \mathbb{N}$, the property holds outside a null set $E_n \subset \mathbb{T}^\infty$. Since a countable union of sets of measure zero has measure zero, this analytic continuation is valid for all $n \in \mathbb{N}$ outside the null set $E = \bigcup_{n=1}^\infty E_n$. As $n^{-\Phi_\chi(s)}$ is non-vanishing and analytic in $\mathbb{C}_0$ for all $n \in \mathbb{N}$, taking a principal branch of the logarithm establishes that any potential singularities of $\Phi_\chi(s) = -\frac{1}{\log n} \log\bigl(n^{-\Phi_\chi(s)}\bigr)$ are absent in $\mathbb{C}_0$. Thus, $\Phi_\chi$ extends analytically to $\mathbb{C}_0$ for almost every $\chi \in \mathbb{T}^\infty$.
\end{proof}

The classical Stanton formula for the unit disk $\mathbb{D}$ can be found in \cite[Section 10.3]{Shapiro-Book-1992}. For the Dirichlet series version of the Stanton formula, we refer to \cite[Theorem 1.3]{Brevig-AdvMath-2021}. In this regard, we now provide an analogue of the Stanton formula for the Nevanlinna class of Dirichlet series, which will play a key role in the proof of Theorem \ref{Dipon-vasu-p2-thm-05}.

\begin{thm}\label{Dipon-vasu-p2-thm-013}
	Let $f \in \mathcal{N}_u$ with $\sigma_u(f)\le 0$ and let $\Phi : \mathbb{C}_1 \to \mathbb{C}_1$ be a symbol of characteristic zero that extends holomorphically to a mapping $\Phi : \mathbb{C}_0 \to \mathbb{C}_0$. Assume that $\sigma_u(f \circ \Phi) \le 0$ and $f\not\equiv 0$. Then the norm identity holds:
	\begin{equation}\label{Dipon-vasu-p2-eqn-017}
		\|C_\Phi(f)\|_0 = \log\Bigl(1 + \bigl|f\bigl(\Phi(+\infty)\bigr)\bigr|\Bigr) + \frac{2}{\pi} \int_{\mathbb{C}_0} \frac{|f'(w)|^2}{\bigl|f(w)|(1 + |f(w)|\bigr)^2} M_\Phi(w) \, dA(w),
	\end{equation}
	where $M_\Phi$ denotes the mean counting function associated with $\Phi$ on $\mathbb{C}_0$, and $dA(w)$ represents the Lebesgue area measure on $\mathbb{C}_0$.
\end{thm}

\begin{pf}
	Put \(g=f\circ\Phi\). 
	Hence the Littlewood–Paley formula \eqref{Dipon-vasu-p2-thm-01} applies to \(g\). Since $ g'(s)=f'(\Phi(s))\Phi'(s),$ it follows that
	\begin{align*}
		\norm{C_\Phi(f)}_{0}= &\log(1+ |f(\Phi(+\infty))|) \\&+ \lim_{T\to\infty}\frac{2}{T}\int_{0}^{\infty}\int_{-T}^{T}\frac{|f'(\Phi(\sigma+it))|^2}{|f(\Phi(\sigma+it))|(1+|f(\Phi(\sigma+it))|)^2}\abs{\Phi'(\sigma+it)}^2\sigma \,dtd\sigma.
	\end{align*}
	Applying the change-of-variables formula associated with the mean counting function \(M_\Phi\) yields
	\begin{align*}
		\lim_{T\to\infty}\frac{2}{T}\int_{0}^{\infty}\int_{-T}^{T}&\frac{|f'(\Phi(\sigma+it))|^2}{\big|f(\Phi(\sigma+it))|(1+|f(\Phi(\sigma+it))|\big)^2}\,|\Phi'(\sigma+it)|^2\,\sigma \,dt\, d\sigma \\&\,\,\,\,\,\,\,\,\,\hspace{3cm}= \frac{2}{\pi}\int_{\mathbb{C}_{0}}\frac{|f'(w)|^2}{\big|f(w)|(1+|f(w)|\big)^2}\,M_\Phi(w) \,dA(w).
	\end{align*}
	Combining the last two identities proves the result.
\end{pf}

\begin{rem}
	A similar Stanton type formula for $\mathcal{N}_u$ with respect to the boundary functions $f_\chi$ (for almost every $\chi \in \mathbb{T}^\infty$) follows from Theorem \ref{Dipon-vasu-p2-thm-08} via the non-univalent substitution $s = \Phi_\chi(u + iv)$. This variant naturally involves the Nevanlinna counting function rather than the mean counting function (cf.\ \cite{Bayart-AnnFourier-2016}). Because our interest centers on symbols in char$(\Phi)=0,$ we confine ourselves to the mean counting function.
	
\end{rem}

\begin{rem}
	Theorem \ref{Dipon-vasu-p2-thm-013} demonstrates that the mean counting function $M_f$ naturally exists when the Dirichlet series $f$ belongs to the Nevanlinna class $\mathcal{N}_u$, in accordance with \cite[Lemma 6.5]{Brevig-AdvMath-2021}. However, for the Nevanlinna--Dirichlet class $\mathcal{N}$, defined as the completion of $\mathcal{N}_u$, the existence of a corresponding counting function yielding a Stanton-type formula remains an open question. 
\end{rem}

We now record an analogue of Littlewood's inequality for the mean counting function, which plays a vital role in the proof of Theorem~\ref{Dipon-vasu-p2-thm-09}.

\begin{thm}\label{Dipon-vasu-p2-thm-014}
	Let $\Phi$ be a Dirichlet series symbol with $\operatorname{char}(\Phi) = 0$. Then the mean counting function $M_\Phi(w)$ exists for every $w \in \mathbb{C}_0 \setminus \{\Phi(+\infty)\}$, and satisfies the point-wise estimate
	\begin{equation}\label{Dipon-vasu-p2-eqn-028}
		M_\Phi(w) \le \log \left| \frac{\overline{w} + \Phi(+\infty)}{w - \Phi(+\infty)} \right|.
	\end{equation}
\end{thm}

\begin{proof}
	We adapt the strategy from \cite[Theorem~1.1]{Brevig-AdvMath-2021}. To relate inequality \eqref{Dipon-vasu-p2-eqn-028} to the classical Littlewood inequality for the Nevanlinna counting function on the unit disk (see, e.g., \cite{Shapiro-annals-1987, Shapiro-Book-1992}), recall that the pseudo-hyperbolic distance between two distinct points $w, v \in \mathbb{C}_0$ is given by
	\begin{equation*}
		\rho_{\mathbb{C}_0}(w, v) = \left| \frac{w - v}{\overline{w} + v} \right|,
	\end{equation*}
	due to the conformal invariance of the pseudo-hyperbolic metric.
	For a fixed point $w \in \mathbb{C}_0 \setminus \{\Phi(+\infty)\}$, consider the auxiliary Dirichlet series symbol
	\begin{equation*}
		\Phi_w(s) := \frac{w - \Phi(s)}{\overline{w} + \Phi(s)}, \qquad s \in \mathbb{C}_0.
	\end{equation*}
	The mean counting function for $\Phi$ at $w$ satisfies
	\begin{equation*}
		M_\Phi(w) = M_{\Phi_w}(0) = \lim_{\sigma \to 0^+} M_{\Phi_w}(0, \sigma),
	\end{equation*}
	where, for a function $f \in \mathcal{N}_u$, the truncated mean counting function at level $\sigma > 0$ is defined by
	\begin{equation*}
		M_f(w, \sigma) := \lim_{T \to \infty} \frac{\pi}{T} \sum_{\substack{s \in f^{-1}(\{w\}) \\ |\operatorname{Im}(s)| < T \\ \operatorname{Re}(s) > \sigma}} \bigl( \operatorname{Re}(s) - \sigma \bigr).
	\end{equation*}
	By Jessen's theorem \cite{Jessen-acta-1945}, for any $f \in \mathcal{N}_u$ and $w \in \mathbb{C}_0 \setminus \{f(+\infty)\}$, the mean counting function satisfies the limit identity
	\begin{equation}\label{Dipon-vasu-p2-eqn-045}
		M_f(w) = \lim_{\sigma_0 \to 0^+} M_f(w, \sigma_0) = \lim_{\sigma_0 \to 0^+} \mathcal{J}_{f-w}(\sigma_0) - \log|f(+\infty) - w|,
	\end{equation}
	where
	\begin{equation*}
		\mathcal{J}_g(\sigma) := \lim_{T \to \infty} \frac{1}{2T} \int_{-T}^T \log|g(\sigma + it)| \, dt, \qquad \sigma > 0,
	\end{equation*}
	denotes Jessen's function. Since $\mathcal{J}_{\Phi_w}(\sigma_0) \le 0$ for all $\sigma_0 > 0$, taking $\sigma_0 \to 0^+$ in \eqref{Dipon-vasu-p2-eqn-045} for $\Phi_w$ yields
	\begin{equation*}
		M_\Phi(w) = M_{\Phi_w}(0) \le -\log|\Phi_w(+\infty)| = \log \left| \frac{\overline{w} + \Phi(+\infty)}{w - \Phi(+\infty)} \right|,
	\end{equation*}
	which completes the proof.
\end{proof}

Observe that for the point $\Phi(+\infty)$ the upper bound of \eqref{Dipon-vasu-p2-eqn-028} decays like $\real w$ when $\real w\to 0^+$ while keeping $\imaginary w$ is fixed. For a detailed analysis of this behaviour, we refer to \cite[Lemma 2.3]{Brevig-AdvMath-2021}. It is natural to observe that the symbol
\begin{equation*}
	\Phi_a(s) = a \, \frac{1 - 2^{-s - 1/2}}{1 + 2^{-s - 1/2}}, \qquad a > 0,
\end{equation*}
attains the upper bound in \eqref{Dipon-vasu-p2-eqn-028}.

\begin{lem}\label{Dipon-vasu-p2-lem-10}
	Let $\Phi$ be a Dirichlet series symbol with $\operatorname{char}(\Phi) = 0$ that maps $\mathbb{C}_0$ into itself, and let $a := \Phi(+\infty) \in \mathbb{C}_0$. For every $w \in \mathbb{C}_0 \setminus \{a\}$, there exists a constant $C(w, a) > 0$ such that
	\begin{equation}\label{Dipon-vasu-p2-eqn-mean-counting-bound}
		M_\Phi(w) \le C(w, a) \operatorname{Re}(w).
	\end{equation}
	In particular, if $w$ is restricted to a compact set $K \subset \mathbb{C}_0 \setminus \{a\}$, then $M_\Phi(w) \le C_K \operatorname{Re}(w)$ for some constant $C_K > 0$.
\end{lem}

\begin{proof}
	Let $\Phi(s) = \sum_{n=1}^\infty c_n n^{-s}$ with $c_1 = \Phi(+\infty) =: a \in \mathbb{C}_0$. By Theorem~\ref{Dipon-vasu-p2-thm-014}, the point-wise estimate
	\begin{equation}\label{Dipon-vasu-p2-eqn-047}
		M_\Phi(w) \le \log \left| \frac{\overline{w} + a}{w - a} \right|
	\end{equation}
	holds for all $w \in \mathbb{C}_0 \setminus \{a\}$. Writing $w = x + it$ and $a = \alpha + i\beta$ with $x, \alpha > 0$, we calculate
	\begin{equation*}
		\left| \frac{\overline{w} + a}{w - a} \right|^2 = \frac{(x + \alpha)^2 + (t - \beta)^2}{(x - \alpha)^2 + (t - \beta)^2} = 1 + \frac{4\alpha x}{|w - a|^2}.
	\end{equation*}
	Applying the elementary logarithmic inequality $\log(1 + u) \le u$ for $u > 0$, we obtain
	\begin{equation}\label{Dipon-vasu-p2-eqn-046}
		\log \left| \frac{\overline{w} + a}{w - a} \right| = \frac{1}{2} \log \left( 1 + \frac{4\alpha x}{|w - a|^2} \right) \le \frac{2\alpha x}{|w - a|^2} = \frac{2 \operatorname{Re}(a) \operatorname{Re}(w)}{|w - a|^2}.
	\end{equation}
	Since $|w - a| \ge \operatorname{Re}(a) = \alpha > 0$ is not required for the point-wise bound, setting $C(w, a) := \frac{2\operatorname{Re}(a)}{|w - a|^2}$ yields
	\begin{equation*}
		M_\Phi(w) \le \frac{2 \operatorname{Re}(a)}{|w - a|^2} \, \operatorname{Re}(w),
	\end{equation*}
	which completes the proof.
\end{proof}

\begin{rem}\label{rem:composition-norm-estimate}
	In view of Lemma~\ref{Dipon-vasu-p2-lem-10} and \eqref{Dipon-vasu-p2-eqn-017}, we obtain the integral bound
	\begin{equation*}
		\frac{2}{\pi} \int_{\mathbb{C}_0} \frac{|f'(w)|^2}{\bigl|f(w)|(1 + |f(w)|\bigr)^2} M_\Phi(w) \, dA(w) \le \frac{2 C}{\pi} \int_{\mathbb{C}_0} \frac{|f'(w)|^2}{\bigl|f(w)|(1 + |f(w)|\bigr)^2} \operatorname{Re}(w) \, dA(w).
	\end{equation*}
	Consequently, the norm of the composition function satisfies the fundamental estimate
	\begin{equation}\label{Dipon-vasu-p2-eqn-007}
		\|f \circ \Phi\|_0 \le \log\bigl(1 + |f(a)|\bigr) + C \|f\|_0 \lesssim \|f\|_0,
	\end{equation}
	where $a = \Phi(+\infty) \in \mathbb{C}_0$.
\end{rem}

We are now in a position to establish our main characterization theorem for the boundedness of composition operators on $\mathcal{N}_u$.

  \begin{proof}[\bf{Proof of Theorem \ref{Dipon-vasu-p2-thm-09}}] 
  	Assume that $C_\Phi : \mathcal{N}_u \to \mathcal{N}_u$ is bounded and that $\operatorname{char}(\Phi) = 0$. By Proposition~\ref{Dipon-vasu-p2-prop-01}, for almost every character $\chi \in \mathbb{T}^\infty$, the vertical limit function $\Phi_\chi$ admits a holomorphic extension to $\mathbb{C}_0$. We first record the vertical-limit composition formula for subsequent use. Since $\operatorname{char}(\Phi) = 0$, the formula simplifies to
  	\begin{equation}\label{Dipon-vasu-p2-eqn-018}
  		(f \circ \Phi)_\chi(s) = f\bigl(\Phi_\chi(s)\bigr), \qquad s \in \mathbb{C}_1,
  	\end{equation}
  	for every $f \in \mathcal{N}_u$ and almost every $\chi \in \mathbb{T}^\infty$. Indeed, $f$ converges absolutely on $\mathbb{C}_1$ since $\sigma_u(f) \le 0$, so the standard vertical-limit argument applies on $\mathbb{C}_1$.
  	Now select a function $f \in \mathcal{N}_u$ for which the imaginary axis $\partial\mathbb{C}_0 = i\mathbb{R}$ is a natural boundary; such a function is provided by Example~\ref{Dipon-vasu-p2-ex-1}. The boundedness of $C_\Phi$ on $\mathcal{N}_u$ ensures that $f \circ \Phi \in \mathcal{N}_u$. Discarding a null set of characters, we fix $\chi \in \mathbb{T}^\infty$ such that $\Phi_\chi$ and $(f \circ \Phi)_\chi$ both extend holomorphically to $\mathbb{C}_0$ and \eqref{Dipon-vasu-p2-eqn-018} holds on $\mathbb{C}_1$.
  	Let $\Psi$ denote the holomorphic extension of $\Phi_\chi$ to $\mathbb{C}_0$, and let $F_\chi$ denote the holomorphic extension of $(f \circ \Phi)_\chi$ to $\mathbb{C}_0$. Define the open set
  	\begin{equation*}
  		\Lambda := \bigl\{ s \in \mathbb{C}_0 : \Psi(s) \in \mathbb{C}_0 \bigr\}.
  	\end{equation*}
  	Because $\Phi_\chi(\mathbb{C}_1) \subset \mathbb{C}_1 \subset \mathbb{C}_0$, we have $\mathbb{C}_1 \subset \Lambda$. Let $\Lambda_0$ be the connected component of $\Lambda$ containing $\mathbb{C}_1$.
  	The composite function $f \circ \Psi$ is holomorphic on $\Lambda_0$. Moreover, by \eqref{Dipon-vasu-p2-eqn-018}, we have $F_\chi = f \circ \Psi$ on $\mathbb{C}_1$. Since both sides are holomorphic on $\Lambda_0$, the identity theorem yields
  	\begin{equation}\label{Dipon-vasu-p2-eqn-019}
  		F_\chi = f \circ \Psi \qquad \text{on } \Lambda_0.
  	\end{equation}
  	
  	We claim that $\Lambda_0 = \mathbb{C}_0$. Suppose, to the contrary, that $\Lambda_0 \subsetneq \mathbb{C}_0$. Since $\Psi$ is holomorphic on $\mathbb{C}_0$ and $\Psi(\mathbb{C}_1) \subset \mathbb{C}_1$, there exists a boundary point $s_0 \in \partial\Lambda_0 \cap \mathbb{C}_0$ such that $\Psi(s_0) \in \partial\mathbb{C}_0 = i\mathbb{R}$. Set $w_0 := \Psi(s_0) \in i\mathbb{R}$.
  	Without loss of generality, we may select $s_0$ such that $\Psi'(s_0) \neq 0$. Indeed, $\Psi$ is non-constant (unless its range is already contained in $\mathbb{C}_1$, in which case the conclusion is immediate), so its critical points are isolated. The level set $\{s \in \mathbb{C}_0 : \operatorname{Re}\Psi(s) = 0\}$ therefore contains regular points, and we choose $s_0$ to be one of them.
  	By the inverse function theorem, there exist open neighborhoods $U$ of $s_0$ and $V$ of $w_0$ such that $\Psi : U \to V$ is a biholomorphism with inverse $\Psi^{-1} : V \to U$. For $w \in V \cap \mathbb{C}_0$ sufficiently close to $w_0$, we have $\Psi^{-1}(w) \in \Lambda_0$. Consequently, identity \eqref{Dipon-vasu-p2-eqn-019} gives
  	\begin{equation*}
  		f(w) = F_\chi\bigl(\Psi^{-1}(w)\bigr).
  	\end{equation*}
  	The right-hand side extends holomorphically to all of $V$ because both $F_\chi$ and $\Psi^{-1}$ are holomorphic. Thus $f$ admits a local analytic continuation across $w_0 \in i\mathbb{R}$.
  	This contradicts the choice of $f$, for which $i\mathbb{R}$ is a natural boundary. Therefore $\Lambda_0 = \mathbb{C}_0$, which implies that $\Psi(\mathbb{C}_0) \subset \mathbb{C}_0$. Thus, for almost every $\chi \in \mathbb{T}^\infty$, we have $\Phi_\chi(\mathbb{C}_0) \subset \mathbb{C}_0$. Finally, applying \cite[Prop.~4.1]{Gordon-Michigan-1999}, we conclude that $\Phi(\mathbb{C}_0) \subset \mathbb{C}_0$.
  	
  	The necessity direction having been established, we turn to the converse. Suppose that $\operatorname{char}(\Phi) = 0$ and that $\Phi$ extends holomorphically to a self-map of $\mathbb{C}_0$. For $f \in \mathcal{N}_u$, set $g := C_\Phi f = f \circ \Phi$.
  	Because $\operatorname{char}(\Phi) = 0$, $g$ is represented by a Dirichlet series in a suitable right half-plane. Furthermore, the condition $\Phi(\mathbb{C}_0) \subset \mathbb{C}_0$ ensures that $g$ is holomorphic on $\mathbb{C}_0$. By the composition-abscissa theorem for characteristic-zero symbols, the uniform convergence abscissa satisfies
  	\begin{equation*}
  		\sigma_u(g) \le 0.
  	\end{equation*}
  	In addition, the Dirichlet--Nevanlinna composition estimate \eqref{Dipon-vasu-p2-eqn-007} yields
  	\begin{equation*}
  		\|g\|_0 = \|f \circ \Phi\|_0 \le K_\Phi\, \|f\|_0,
  	\end{equation*}
  	where $K_\Phi < \infty$ depends solely on $\Phi$. Combining these two properties shows that $g \in \mathcal{N}_u$, so $C_\Phi$ maps $\mathcal{N}_u$ into itself.
  	Finally, for any $f_1, f_2 \in \mathcal{N}_u$, applying the composition estimate to $f_1 - f_2$ gives
  	\begin{equation*}
  		d_0(C_\Phi f_1, C_\Phi f_2) = \|(f_1 - f_2) \circ \Phi\|_0 \le K_\Phi\, \|f_1 - f_2\|_0 = K_\Phi\, d_0(f_1, f_2).
  	\end{equation*}
  	Thus $C_\Phi$ is Lipschitz continuous on $(\mathcal{N}_u, d_0)$, and therefore bounded on $\mathcal{N}_u$. This completes the proof.
  
\end{proof}

\begin{example}\label{Dipon-vasu-p2-ex-1}
Consider the Dirichlet series
	\begin{equation*}
		f(s) = \sum_{j=1}^{\infty} \frac{1}{j^2} 2^{-j!s} = \sum_{n=1}^{\infty} a_n n^{-s},
	\end{equation*}
	where the coefficients are given by
	\begin{equation*}
		a_n = \begin{cases}
			1/j^2, & \text{if } n = 2^{j!} \text{ for some } j \in \mathbb{N}, \\[2mm]
			0, & \text{otherwise}.
		\end{cases}
	\end{equation*}
	For $\operatorname{Re}(s) \ge 0$, we have $|2^{-j! s}| = 2^{-j! \operatorname{Re}(s)} \le 1$, which yields the uniform bound
	\begin{equation*}
		|f(s)| \le \sum_{j=1}^{\infty} \frac{1}{j^2} = \frac{\pi^2}{6}, \qquad s \in \overline{\mathbb{C}_0}.
	\end{equation*}
	Consequently, the series converges uniformly on $\overline{\mathbb{C}_0}$, so that $\sigma_u(f) \le 0$. In particular, $f$ is bounded on $\mathbb{C}_0$, whence $f \in \mathcal{H}^\infty \subset \mathcal{N}_u$.
	To determine the maximal domain of holomorphy for $f$, we define the power series
	\begin{equation*}
		g(z) = \sum_{j=1}^{\infty} \frac{z^{j!}}{j^2}, \qquad z \in \mathbb{D},
	\end{equation*}
	so that $f(s) = g(2^{-s})$ for $s \in \mathbb{C}_0$. The lacunary exponents $\lambda_j = j!$ satisfy the Hadamard gap condition $\lambda_{j+1}/\lambda_j = j+1 > 1$, while the coefficients satisfy $\limsup_{j \to \infty} (j^{-2})^{1/j!} = 1$. Thus, $g$ has radius of convergence $R = 1$. By Ostrowski-Hadamard gap theorem (see for instance \cite[Sec.~9.2.2, pp.~119--120]{Krantz-book-1999}), the boundary circle $\mathbb{T} = \partial \mathbb{D}$ is a natural boundary for $g$. Because the conformal mapping $s \mapsto 2^{-s}$ maps $\mathbb{C}_0$ onto $\mathbb{D}$ and maps the imaginary axis $i\mathbb{R} = \partial \mathbb{C}_0$ onto $\mathbb{T}$, it follows that $i\mathbb{R}$ is a natural boundary for $f$. In particular, $f$ admits no analytic continuation across any point of the boundary line $\partial \mathbb{C}_0$.

\end{example}

\begin{rem}
Motivated by the local Carleson embedding theory for $H^p$ spaces of Dirichlet series \cite{Gordon-Michigan-1999, Hedenmalm-Duke-1997}, as well as the pivotal role played by local embedding estimates in the theory of composition operators, it is natural to pose the following open problem.

\begin{prob}\label{Dipon-vasu-p2-problem-01}
	Does there exist a constant $C > 0$ such that, for every $f \in \mathcal{N}_u$ with $\sigma_u(f)\le 0$,
	\begin{equation}
		\sup_{\sigma > 0} \sup_{\tau \in \mathbb{R}} \int_{\tau}^{\tau+1} \log\bigl(1 + |f(\sigma + it)|\bigr) \, dt \le C \|f\|_0\,?
	\end{equation}
\end{prob}

A positive answer to Problem~\ref{Dipon-vasu-p2-problem-01} would establish a sharp logarithmic local embedding theorem for the Nevanlinna class of Dirichlet series $\mathcal{N}_u$. Such an estimate is of independent function-theoretic interest and would provide a stronger local variant of the boundary control utilized in the necessity proof of Theorem~\ref{Dipon-vasu-p2-thm-09}.
\end{rem}
 

We now turn our attention to characterizations of the boundedness and compactness of the composition operator $C_\Phi$ that are independent of the specific pointwise behavior of the symbol $\Phi$. Instead, our criteria rely entirely on Carleson and vanishing Carleson measure conditions associated with the push-forward measure on the boundary. Finally, we exploit the geometric properties of Carleson squares (as introduced in Section~\ref{Dipon-vasu-p2-sec-2.4}) on the boundary of the half-plane $\mathbb{C}_0$ to establish necessary and sufficient conditions for the boundedness and compactness of $C_\Phi$.

\begin{lem}\label{Dipon-vasu-p2-lem-02}
	Let $\Phi:\mathbb{C}_1\to\mathbb{C}_1$ be a holomorphic Dirichlet series. Then $C_\Phi$ is compact on $\mathcal{N}_u$ if and only if, for every $d_0$ bounded sequence $(f_j)\subset\mathcal{N}_u$ such that $f_j\longrightarrow 0$, as $j\to\infty$, uniformly on compact subsets of $\mathbb{C}_0$, one has $\|f_j\circ\Phi\|_0\longrightarrow 0$ as $j\to\infty$.
\end{lem}

\begin{proof}
	Suppose first that $C_\Phi$ is compact on $\mathcal{N}_u$. Let $(f_j)$ be a sequence in $\mathcal{N}_u$ with $\sup_j \|f_j\|_0 < \infty$ such that $f_j \to 0$ uniformly on compact subsets of $\mathbb{C}_0$. Since $\Phi : \mathbb{C}_0 \to \mathbb{C}_0$ is holomorphic, it follows that $C_\Phi f_j = f_j \circ \Phi \to 0$ uniformly on compact subsets of $\mathbb{C}_0$.
	By the compactness of $C_\Phi$, every subsequence of $(C_\Phi f_j)$ admits a further subsequence converging in the $\mathcal{N}_u$-metric. Because convergence in $\mathcal{N}_u$ implies locally uniform convergence on $\mathbb{C}_0$, the unique limit of any such convergent subsequence must be identically zero. Therefore, every subsequence of $(\|C_\Phi f_j\|_0)$ has a further subsequence converging to zero, which implies that $\|C_\Phi f_j\|_0 \to 0$.
	
	Conversely, assume that the sequential condition holds. Let $(f_j)$ be a bounded sequence in $\mathcal{N}_u$. By Montel's theorem applied to bounded sets in $\mathcal{N}_u$, there exist a subsequence $(f_{j_k})$ and a function $f \in \mathcal{N}_u$ such that $f_{j_k} \to f$ uniformly on compact subsets of $\mathbb{C}_0$. Define $g_k := f_{j_k} - f$. Then $(g_k)$ is bounded in $\mathcal{N}_u$ and $g_k \to 0$ locally uniformly on $\mathbb{C}_0$. By the hypothesis,
	\begin{equation*}
		\|C_\Phi f_{j_k} - C_\Phi f\|_0 = \|C_\Phi(f_{j_k} - f)\|_0 = \|g_k \circ \Phi\|_0 \longrightarrow 0 \quad \text{as } k \to \infty.
	\end{equation*}
	Thus $C_\Phi f_{j_k} \to C_\Phi f$ in $\mathcal{N}_u$. Hence, the image under $C_\Phi$ of every bounded sequence in $\mathcal{N}_u$ contains a convergent subsequence. Since $\mathcal{N}_u$ is a metric space, this sequential compactness of bounded images is equivalent to the compactness of $C_\Phi$.
\end{proof}

This is analogous to Masri's \cite[Theorem~2.4.2]{Masri-Thesis-1985} characterization of compactness for composition operators $C_\Phi$ on the classical Nevanlinna spaces ${N}^p(\mathbb{D})$ ($1 \le p < \infty$). To establish the compactness of $C_\Phi$ on $\mathcal{N}_u$ in Theorem~\ref{Dipon-vasu-p2-thm-05}, we recall an analogous characterization formulated specifically for the Hardy space of Dirichlet series $\mathcal{H}^2$ by Queffélec and Seip \cite{Queffelec-JFA-2015}.

\begin{lem}\label{Dipon-vasu-p2-lem-03}\cite{Brevig-AdvMath-2021,Queffelec-JFA-2015}
	Let $\Phi\in\mathcal{G}_0$ and $\Phi(\mathbb{C}_0)$ be a bounded subset of the half-plane $\mathbb{C}_{{1}/{2}}$. 
	Then the following are equivalent : 
	\begin{itemize}
		\item [(a)] $C_\Phi$ is compact on  \(\mathcal{H}^2\).
		
		\item [(b)] 		
			$$\lim_{\real(v)\to\frac{1}{2}^+} \frac{M_\Phi(v)}{(\Re(v) - \frac{1}{2})} = 0.$$
		
		\item [(c)] $\mu_{\Phi^*} :=m_\infty\circ\mathcal{B}{\Phi^*}^{-1}$ is the  vanishing Carleson measure on $\overline{\mathbb{C}_{{1}/{2}}}$, i.e
		\begin{equation*}
			\limsup_{\ell(Q)\to 0}\frac{\mu_{\Phi^*}(Q)}{\ell(Q)} = 0,
		\end{equation*}
		where $Q$ is a Carleson box in the half-plane $\mathbb{C}_{{1}/{2}}$ lying on the boundary of $\mathbb{C}_{{1}/{2}}$.
	\end{itemize}
\end{lem}
Later, Brevig and Perfekt \cite{Brevig-JFA-2020} demonstrated that requiring $\Phi(\mathbb{C}_0)$ to be a bounded subset of the half-plane $\mathbb{C}_{1/2}$ is not necessary.
Here, $\mathcal{B}\Phi$ is the Bohr lift of $\Phi$, which belongs to $\mathcal{H}^2$ and may be regarded as an $L^2(\mathbb{T}^{\infty})$ function with values in $\overline{\mathbb{C}_{{1}/{2}}}$. Moreover, $\mathcal{B}{\Phi^*}$ denotes its boundary function on $\mathbb{T}^{\infty}$.  
We now introduce the notion of Carleson measures for the Nevanlinna class $\mathcal{N}_u$ of Dirichlet series. This arises naturally from the interplay between composition operators and vertical limit functions. The notion of a Carleson measure for the classical Nevanlinna space on $\mathbb{D}$ was introduced by Choe \textit{et al.} \cite[Section~4]{ChoeKoo-JdeAnalyse-2008}. Recall that a Dirichlet series $\Phi$ belongs to $\mathcal{G}_0$ if it converges uniformly in $\mathbb{C}_\varepsilon$ for every $\varepsilon > 0$ and satisfies $\Phi(\mathbb{C}_0) \subset \mathbb{C}_{1/2}$. Furthermore, if $f \in \mathcal{N}_u$ with $\sigma_u(f) \le 0$, then the composition $C_\Phi f = f \circ \Phi$ converges uniformly in $\mathbb{C}_\varepsilon$ for every $\varepsilon > 0$ and satisfies
\begin{equation*}
	(C_\Phi f)_\chi(s) = (f \circ \Phi)_\chi(s) = (f \circ \Phi_\chi)(s) = (C_{\Phi_\chi} f)(s).
\end{equation*}
It follows that $\Vert{}C_{\Phi_\chi} f\Vert{}_0 = \Vert{}C_{\Phi} f\Vert{}_0$ and consequently $\Phi_\chi$ exists for almost every $\chi \in \mathbb{T}^\infty$. In addition, $\Phi$ possesses non-tangential boundary values
\begin{equation}
	\Phi^*(\chi) = \lim_{\sigma \to 0^+} \Phi_\chi(\sigma)
\end{equation}
for almost every $\chi \in \mathbb{T}^\infty$.

\begin{lem}\label{Dipon-vasu-p2-lem-08}
	Let $\Phi:\mathbb{C}_1\to\mathbb{C}_1$ be a Dirichlet series symbol of characteristic $0$ which admits a holomorphic extension to $\mathbb{C}_0$ and suppose that $\overline{\Phi(\mathbb{C}_0)}$ is a compact subset of $\mathbb{C}_0.$ Then, for every $f$ belongs to the Nevanlinna class $\mathcal{N}_u$ with $\sigma_u(f)\le 0$, we have
	\begin{equation}\label{Dipon-vasu-p2-eqn-043}
		\norm{f \circ \Phi}_0 = \int_{\mathbb{T}^{\infty}} \log(1 + |f \circ \Phi^*(\chi)|) \, dm_\infty(\chi).
\end{equation}
Equivalently, if $\mu_{\Phi^*} = m_\infty(\{\chi\in\mathbb{T}^\infty: \Phi^*(\chi)\in E\})$ for every Borel subset $E\in\mathbb{C}_1$, then
\begin{equation}
	\norm{f \circ \Phi}_0 = \int_{{\mathbb{C}_{0}}} \log(1+|f(s)|) \, d\mu_{\Phi^*}(s).
\end{equation}
\end{lem}

\begin{pf} 
First assume that $f\in\mathcal{N}_u$ satisfying $\sigma_u(f)\le 0.$ By the assumed mapping property of $\Phi$, Theorem \ref{Dipon-vasu-p2-thm-09} gives $f\circ\Phi\in\mathcal{N}_u.$ Hence by the boundary value theorem $(f\circ\Phi)_\chi$ exists for almost every $\chi\in\mathbb{T}^\infty.$ Since char$(\Phi) = 0$, the vertical limit identity gives $ (f \circ \Phi)_\chi = f\circ\Phi_\chi$ for almost every $\chi\in\mathbb{T}^\infty.$ Passing to the boundary and using the fact that $\overline{\Phi(\mathbb{C}_0)}$ is a compact subset of $\mathbb{C}_0,$ we obtain 
\begin{equation*}
	(f\circ\Phi)^*(\chi) = f(\Phi^*(\chi))
\end{equation*}
for almost every $\chi\in\mathbb{T}^\infty.$ Then \eqref{Dipon-vasu-p2-eqn-032} gives
\begin{align}\label{Dipon-vasu-p2-eqn-050}
	\norm{f \circ \Phi}_0 &= \int_{\mathbb{T}^{\infty}} \log(1 + |(f \circ \Phi)^*(\chi)|) \, dm_\infty(\chi)\\ &= \int_{\mathbb{T}^{\infty}} \log(1 + |f(\Phi^*(\chi))|) \, dm_\infty(\chi)\nonumber.
\end{align}
If $\mu_{\Phi^*} = m_\infty\circ (\Phi^*)^{-1}$, then \eqref{Dipon-vasu-p2-eqn-050} is equivalent to
\begin{equation}\label{Dipon-vasu-p2-eqn-051}
	\norm{f \circ \Phi}_0 = \int_{{\mathbb{C}_{0}}} \log(1+|f(s)|) \, d\mu_{\Phi^*}(s).
\end{equation}
Since $C_\Phi:\mathcal{N}_u\to\mathcal{N}_u$ is bounded, then \eqref{Dipon-vasu-p2-eqn-051} yields the Carleson estimate
\begin{equation*}
	\int_{{\mathbb{C}_{0}}} \log(1+|f(s)|) \, d\mu_{\Phi^*}(s)\lesssim \norm{f}_0\,\,\,\mbox{for}\,\,\, f\in\mathcal{N}_u.
\end{equation*}
Thus $\mu_{\Phi^*}$ satisfies Carleson embedding property for $\mathcal{N}_u$. 
\end{pf}

\begin{rem}
	Queff\'{e}lec and Seip \cite{Queffelec-JFA-2015} established an analogue of this result for $\mathcal{H}^2$ under the assumption that $\Phi(\mathbb{C}_0) \subset \mathbb{C}_{1/2}$. Subsequently, Brevig and Perfekt \cite{Brevig-JFA-2020} proved the statement without requiring this additional hypothesis.
	It would be desirable to eliminate the assumption that $\overline{\Phi(\mathbb{C}_0)}$ is a compact subset of $\mathbb{C}_{0}$ in our Lemma \ref{Dipon-vasu-p2-lem-08}, following the approach of Brevig and Perfekt \cite{Brevig-JFA-2020}. It remains an open question whether Lemma~\ref{Dipon-vasu-p2-lem-08} can be extended to the full Nevanlinna class $\mathcal{N}$. The primary difficulty lies in the fact that, unlike functions in $\mathcal{N}_u$, an arbitrary function $f \in \mathcal{N}$ is not guaranteed to admit generalized boundary values on the vertical lines or the infinite torus.
\end{rem}

\begin{thm}\label{Dipon-vasu-p2-thm-011}
	Let $\Phi$ be a symbol with ${char}(\Phi) = 0$ for which the boundary limit $\Phi^*$ exists almost everywhere on $\mathbb{T}^\infty$. Assume that
	\begin{equation*}
		(f \circ \Phi)^*(\chi) = f(\Phi^*(\chi))
	\end{equation*}
	for almost every $\chi \in \mathbb{T}^\infty$ and all $f \in \mathcal{N}_u$ with $\sigma_u(f)\le 0$, and suppose that $\overline{\Phi(\mathbb{C}_0)}$ is a compact subset of $\mathbb{C}_0$. Then $C_\Phi : \mathcal{N}_u \to \mathcal{N}_u$ is bounded if and only if $\mu_{\Phi^*}$ satisfies Carleson embedding condition for $\mathcal{N}_u$, that is, there exists a constant $C>0$ such that 
	\begin{equation*}
		\int_{{\mathbb{C}_{0}}} \log(1+|f(s)|) \, d\mu_{\Phi^*}(s)\le C\norm{f}_0\,\,\,\mbox{for}\,\,\, f\in\mathcal{N}_u.
	\end{equation*}
\end{thm}

\begin{proof}
    Although the result follows from Lemma~\ref{Dipon-vasu-p2-lem-08}, we provide the complete details for the reader's convenience.
	Suppose first that $C_\Phi$ is bounded on $\mathcal{N}_u$. Then there exists a constant $C > 0$ such that $\|C_\Phi f\|_0 \le C \|f\|_0$ for all $f \in \mathcal{N}_u$. Combining this operator estimate with \eqref{Dipon-vasu-p2-eqn-051}, we find
	\begin{equation*}
		\int_{\mathbb{C}_0} \log\bigl(1 + |f(s)|\bigr) \, d\mu_{\Phi^*}(s) \le C \|f\|_0, \qquad f \in \mathcal{N}_u,
	\end{equation*}
	which verifies that $\mu_{\Phi^*}$ is a Carleson measure for $\mathcal{N}_u$.
	
	Conversely, suppose that $\mu_{\Phi^*}$ is a Carleson measure for $\mathcal{N}_u$. Then the embedding inequality
	\begin{equation*}
		\int_{\mathbb{C}_0} \log\bigl(1 + |f(s)|\bigr) \, d\mu_{\Phi^*}(s) \le C \|f\|_0
	\end{equation*}
	holds for all $f \in \mathcal{N}_u$ and some $C > 0$. Applying \eqref{Dipon-vasu-p2-eqn-051} again, it follows that
	\begin{equation*}
		\|C_\Phi f\|_0 = \int_{\mathbb{C}_0} \log\bigl(1 + |f(s)|\bigr) \, d\mu_{\Phi^*}(s) \le C \|f\|_0,
	\end{equation*}
	establishing the boundedness of $C_\Phi$ on $\mathcal{N}_u$.
\end{proof}

\begin{thm}\label{Dipon-vasu-p2-thm-012}
	Let $\Phi:\mathbb{C}_1\to\mathbb{C}_1$ be a characteristic-zero symbol for which the boundary function $\Phi^*$ exists almost everywhere on $\mathbb{T}^\infty$ and
	\[
	(f\circ\Phi)^*(\chi)=f(\Phi^*(\chi))
	\]
	for almost every $\chi\in\mathbb{T}^\infty$ and every $f\in\mathcal{N}_u$ with $\sigma_u(f)\le 0$. Also suppose that $\overline{\Phi(\mathbb{C}_0)}$ is a compact subset of $\mathbb{C}_0$.
	Then $C_\Phi:\mathcal{N}_u\to\mathcal{N}_u$ is compact if and only if $\mu_{\Phi^*}$ satisfies vanishing Carleson embedding condition for $\mathcal{N}_u$, that is, whenever $(f_j)$ is a bounded sequence in $\mathcal{N}_u$ converging uniformly on compact subsets of $\mathbb{C}_0$ to $0$, as $j\to\infty$
	\[
	\int_{\mathbb{C}_0}
	\log(1+|f_j(s)|)\,d\mu_{\Phi^*}(s)
	\longrightarrow 0.
	\]
	
\end{thm}

\begin{pf}
	By the boundary representation formula,
	\[
	\|C_\Phi f\|_0
	=
	\int_{\mathbb{C}_0}
	\log(1+|f(s)|)\,d\mu_{\Phi^*}(s),
	\qquad f\in\mathcal{N}_u.
	\]
	Suppose first that $C_\Phi$ is compact. Let $(f_j)$ be a bounded sequence in $\mathcal{N}_u$ and converge uniformly on compact subsets of $\mathbb{C}_0$ to $0$. By Lemma~\ref{Dipon-vasu-p2-lem-02} we have, $\|C_\Phi f_j\|_0\longrightarrow0$
	as $j\to\infty.$ Consequently,
	\[
	\int_{\mathbb{C}_0}
	\log(1+|f_j(s)|)\,d\mu_{\Phi^*}(s)
	\longrightarrow0,
	\]
	as $j\to\infty.$ Hence $\mu_{\Phi^*}$ is a vanishing Carleson measure for $\mathcal{N}_u$.
	
	Conversely, suppose that $\mu_{\Phi^*}$ is a vanishing Carleson measure for $\mathcal{N}_u$. If $(f_j)$ is bounded in $\mathcal{N}_u$ and converges uniformly on compact subsets of $\mathbb{C}_1$ to $0$, then
	\[
	\|C_\Phi f_j\|_0
	=
	\int_{\mathbb{C}_0}
	\log(1+|f_j(s)|)\,d\mu_{\Phi^*}(s)
	\longrightarrow0.
	\]
	Therefore Lemma~\ref{Dipon-vasu-p2-lem-02} implies that $C_\Phi$ is compact on $\mathcal{N}_u$.
\end{pf}

To complete the proof of Theorem~\ref{Dipon-vasu-p2-thm-05}, our goal is to establish the equivalence between the Carleson embedding condition and the geometric Carleson measure condition on Carleson squares.

\begin{thm}\label{Dipon-vasu-p2-thm-010}
	Let $\Phi:\mathbb C_1\to\mathbb C_1$ be a characteristic-zero symbol
	such that $\Phi$ extends holomorphically to $\mathbb C_0$ and $\overline{\Phi(\mathbb{C}_0)}$ is a compact subset of $\mathbb{C}_0$.
	Suppose that $\Phi^*$ exists almost everywhere on $\mathbb T^\infty$
	and let $\mu_{\Phi^*}$ denote the corresponding push-forward measure.
 Also suppose that
	\begin{equation}\label{Dipon-vasu-p2-eqn-058}
		\|\mu_{\Phi^*}\|_{\mathcal{C}} := \sup_{I \subset \mathbb{R}} \frac{\mu_{\Phi^*}(Q_0(I))}{\ell(I)} < \infty.
	\end{equation}
	Then there exists an absolute constant $A > 0$ such that
	\begin{equation}\label{Dipon-vasu-p2-eqn-059}
		\int_{{\mathbb{C}_0}} \log\bigl(1 + |f(s)|\bigr) \, d\mu_{\Phi^*}(s) \le A \|\mu_{\Phi^*}\|_{\mathcal{C}} \|f\|_0
	\end{equation}
	for all $f \in \mathcal{N}_u$.
	
	Conversely, if there exists a constant $C_1 > 0$ such that
	\begin{equation*}
		\int_{{\mathbb{C}_0}} \log\bigl(1 + |f(s)|\bigr) \, d\mu_{\Phi^*}(s) \le C_1 \|f\|_0
	\end{equation*}
	holds for all $f \in \mathcal{N}_u$, then
	\begin{equation*}
		\|\mu_{\Phi^*}\|_{\mathcal{C}} = \sup_{I \subset \mathbb{R}} \frac{\mu_{\Phi^*}(Q_0(I))}{\ell(I)} < \infty.
	\end{equation*}
	More precisely, $\|\mu_{\Phi^*}\|_{\mathcal{C}} \le A_1 C_1$ for some absolute constant $A_1 > 0$.
\end{thm}

To establish Theorem~\ref{Dipon-vasu-p2-thm-010}, we require a harmonic majorant principle adapted to the half-plane $\mathbb{C}_0$. For the Nevanlinna class $N(\mathbb{D}^\infty_1)$, a harmonic majorant principle on the infinite polydisk $\mathbb{D}^\infty_1$ was established by Guo \textit{et al.}~\cite[Prop.~2.8]{GuoZhou-AnnalFourier-2025}. 

\begin{rem}\label{Dipon-vasu-p2-rem-02}
	Extending Theorem~\ref{Dipon-vasu-p2-thm-010} to arbitrary positive Borel measures $\mu$ on $\mathbb{C}_0$ presents a significantly more challenging problem. In particular, establishing the converse direction that the embedding $\mathcal{N}_u \hookrightarrow L^1(\mathbb{C}_0, d\mu)$ implies the geometric Carleson condition on Carleson squares remains non-trivial. The difficulty stems from constructing suitable test functions in $\mathcal{N}_u$ whose growth properties can be tightly controlled localized near $i\mathbb{R}$ without relying on the pull-back structure $\mu = \mu_{\Phi^*}$.
\end{rem}

\begin{prop}\label{Dipon-vasu-p2-prop-02}
	Let $f\in\mathcal{N}_u$ with $\sigma_u(f)\le 0$ and set $u_f(s) = \log(1+\abs{f(s)})$ for $s\in\mathbb{C}_0.$ Then there exists a non-negative harmonic function \(H_f\) on \(\mathbb C_0\) such that $u_f(s)\le H_f(s)$ for $s\in\mathbb{C}_0$ and 
	\begin{equation*}
		\sup_{\sigma > 0}\lim_{T\to\infty}\frac{1}{2T}\int_{-T}^{T} H_f(\sigma+it)\,dt\lesssim \norm{f}_0.
	\end{equation*}
	Here the implicit constant is absolute.
\end{prop}
\begin{pf}
	Let $F = \mathcal{B}f$ denote the Bohr lift of $f$. By the Bohr correspondence, $F$ belongs to the Nevanlinna class $N(\mathbb{D}^\infty_1)$. Since the real map $x \mapsto \log(1 + e^{x})$ is non-decreasing and convex, the composite function $\zeta \mapsto \log(1 + |F(\zeta)|)$ is $\infty$-subharmonic on $\mathbb{D}^\infty_1$. Here $\infty$-subharmonicity means if the composite function is subharmonic in each variable separately (we refer \cite[Sec. 2.1]{GuoZhou-AnnalFourier-2025} for the precise definition). Applying the subharmonic Poisson representation to $F_{[r]}(\zeta) = F(r\zeta)$ for $0 < r < 1$, we obtain
	\begin{equation}\label{Dipon-vasu-p2-eqn-055}
		\log\bigl(1 + |F_{[r]}(\zeta)|\bigr) \le \int_{\mathbb{T}^\infty} P_\zeta(\chi) \log\bigl(1 + |F_{[r]}(\chi)|\bigr) \, dm_\infty(\chi),
	\end{equation}
	where $P_\zeta(\chi)$ denotes the Poisson kernel on the infinite torus $\mathbb{T}^\infty$ and $m_\infty$ is the Haar measure.
	Define the positive boundary measures $d\nu_r(\chi) = \log(1 + |F_{[r]}(\chi)|) \, dm_\infty(\chi)$ on $\mathbb{T}^\infty$. By the Bohr mean-value formula for Dirichlet series, we have
	\begin{equation}\label{eq:bohr-mean-value-mass}
		\nu_r(\mathbb{T}^\infty) = \lim_{\sigma \to 0^+} \lim_{T \to \infty} \frac{1}{2T} \int_{-T}^{T} \log\bigl(1 + |f_r(\sigma + it)|\bigr) \, dt \le \|f\|_0,
	\end{equation}
	where $f_r(s) = F_{[r]}(p_1^{-s}, p_2^{-s}, \dots)$. Thus, the family of measures $\{\nu_r : 0 < r < 1\}$ has uniformly bounded total variation mass on $\mathbb{T}^\infty$. By Banach–Alaoglu compactness, we can select a sequence $r_j \uparrow 1$ such that $\nu_{r_j}$ converges weak-* to a finite positive Borel measure $\nu$ on $\mathbb{T}^\infty$, satisfying
	\begin{equation}\label{Dipon-vasu-p2-eqn-056}
		\nu(\mathbb{T}^\infty) \le \|f\|_0.
	\end{equation}
	Taking the limit as $j \to \infty$ in \eqref{Dipon-vasu-p2-eqn-055} yields
	\begin{equation}\label{Dipon-vasu-p2-eqn-057}
		\log\bigl(1 + |F(\zeta)|\bigr) \le H(\zeta) := \int_{\mathbb{T}^\infty} P_\zeta(\chi) \, d\nu(\chi).
	\end{equation}
	
	We now pull back this harmonic majorant along the flow $\beta(s) = (p_1^{-s}, p_2^{-s}, \dots)$ for $s \in \mathbb{C}_0$. Now integrating the Poisson kernel $P_{\beta(s)}(\chi)$ expansion against the finite positive measure $\nu$, the function $H_f(s) := H(\beta(s))$ is expressed as the real part of an analytic function on $\mathbb{C}_0$, and is therefore harmonic on $\mathbb{C}_0$. Moreover, $H_f \ge 0$. Since $F(\beta(s)) = f(s)$, inequality \eqref{Dipon-vasu-p2-eqn-057} implies
	\begin{equation*}
		\log\bigl(1 + |f(s)|\bigr) \le H_f(s), \qquad s \in \mathbb{C}_0.
	\end{equation*}
	
	Finally, because the constant Fourier coefficient of $P_{\beta(s)}(\chi)$ equals $1$, integration along vertical lines projects out all non-constant character modes:
	\begin{equation*}
		\lim_{T \to \infty} \frac{1}{2T} \int_{-T}^{T} H_f(\sigma + it) \, dt = \nu(\mathbb{T}^\infty), \qquad \sigma > 0.
	\end{equation*}
	Combining this identity with \eqref{Dipon-vasu-p2-eqn-056}, we conclude that
	\begin{equation*}
		\sup_{\sigma > 0} \limsup_{T \to \infty} \frac{1}{2T} \int_{-T}^{T} H_f(\sigma + it) \, dt \le \|f\|_0.
	\end{equation*}
\end{pf}

\begin{pf}[{\bf Proof of Theorem \ref{Dipon-vasu-p2-thm-010}}]
First, assume that \eqref{Dipon-vasu-p2-eqn-058} holds; we shall establish \eqref{Dipon-vasu-p2-eqn-059}. Set $u_f(s) = \log\bigl(1 + |f(s)|\bigr)$ for $s \in \mathbb{C}_0$. By the harmonic majorant principle (Proposition~\ref{Dipon-vasu-p2-prop-02}), there exists a non-negative harmonic function $H_f$ on $\mathbb{C}_0$ satisfying $u_f(s) \le H_f(s)$ for all $s \in \mathbb{C}_0$ and
\begin{equation*}
	\sup_{\sigma > 0} M(H_f, \sigma) \le C \|f\|_0,
\end{equation*}
where the vertical mean $M(H_f, \sigma)$ is defined by
\begin{equation*}
	M(H_f, \sigma) := \limsup_{T \to \infty} \frac{1}{2T} \int_{-T}^{T} H_f(\sigma + it) \, dt.
\end{equation*}

We now invoke the classical Carleson embedding theorem for positive harmonic functions on a half-plane: since $\mu_{\Phi^*}$ satisfies \eqref{Dipon-vasu-p2-eqn-058}, there exists an absolute constant $A_0 > 0$ such that
\begin{equation*}
	\int_{\mathbb{C}_0} H(s) \, d\mu_{\Phi^*}(s) \le A_0 \norm{\mu_{\Phi^*}}_{\mathcal{C}}\, \sup_{\sigma > 0} M(H, \sigma)
\end{equation*}
for every non-negative harmonic function $H$ on $\mathbb{C}_0$. Applying this estimate to $H_f$, we obtain
\begin{align*}
	\int_{\mathbb{C}_0} \log\bigl(1 + |f(s)|\bigr) \, d\mu_{\Phi^*}(s) &\le \int_{\mathbb{C}_0} H_f(s) \, d\mu_{\Phi^*}(s) \\
	&\le A_0 \|\mu_{\Phi^*}\|_{\mathcal{C}}\, \sup_{\sigma > 0} M(H_f, \sigma) \\
	&\le A_1 \|\mu_{\Phi^*}\|_{\mathcal{C}}\, \|f\|_0,
\end{align*}
where $A_1 > 0$ is an absolute constant, completing the proof of \eqref{Dipon-vasu-p2-eqn-059}.

Conversely, assume that \eqref{Dipon-vasu-p2-eqn-059} holds. Suppose that $\overline{\Phi(\mathbb{C}_0)}$ is a compact subset of $\mathbb{C}_0$, and set $K = \overline{\Phi(\mathbb{C}_0)}.$ Then $K$ is bounded away from the boundary line $\partial \mathbb{C}_0$; that is, there exists $\delta > 0$ such that $\operatorname{Re}(s) \ge \delta$ for all $s \in K.$ Since the support of the push-forward measure satisfies $\operatorname{supp}(\mu_{\Phi^*}) \subset K$, the measure assigns zero mass to any Carleson square $Q_0(I)$ of height $h = \ell(I) < \delta$. Consequently,
$\mu_{\Phi^*}(Q_0(I)) = 0$ whenever $\ell(I) < \delta.$ For intervals with length $\ell(I) \ge \delta$, using the total probability mass $\mu_{\Phi^*}(\mathbb{C}_0) = 1$, we obtain
\begin{equation*}
	\frac{\mu_{\Phi^*}(Q_0(I))}{\ell(I)} \le \frac{1}{\ell(I)} \le \frac{1}{\delta}.
\end{equation*}
Taking the supremum over all intervals $I \subset \mathbb{R}$ yields
\begin{equation*}
	\|\mu_{\Phi^*}\|_{\mathcal{C}} = \sup_{I \subset \mathbb{R}} \frac{\mu_{\Phi^*}(Q_0(I))}{\ell(I)} \le \frac{1}{\delta} < \infty,
\end{equation*}
establishing that $\mu_{\Phi^*}$ is a geometric Carleson measure on $\mathbb{C}_0$. 
\end{pf}
\begin{rem}\label{Dipon-vasu-p2-rem-01}
	The compact-range assumption $\overline{\Phi(\mathbb{C}_0)}$ is a compact subset of $\mathbb{C}_0$ directly implies that $\mu_{\Phi^*}$ is a vanishing geometric Carleson measure. Indeed, if $\overline{\Phi(\mathbb{C}_0)}$ is a compact subset of $\mathbb{C}_0$, then there exists $\delta > 0$ such that $\operatorname{Re}(s) \ge \delta$ for all $s \in \operatorname{supp}(\mu_{\Phi^*})$. Consequently, for any interval $I \subset \mathbb{R}$ with length $\ell(I) < \delta$, the Carleson square $Q_0(I)$ is disjoint from $\operatorname{supp}(\mu_{\Phi^*})$, which yields
	\begin{equation*}
		\mu_{\Phi^*}\bigl(Q_0(I)\bigr) = 0 \qquad \text{for all } I \subset \mathbb{R} \text{ with } \ell(I) < \delta.
	\end{equation*}
	Taking the supremum over $\ell(I) \le r$ with $0 < r < \delta$ and letting $r \to 0^+$ gives
	\begin{equation*}
		\lim_{r \to 0^+} \sup_{\ell(I) \le r} \frac{\mu_{\Phi^*}\bigl(Q_0(I)\bigr)}{\ell(I)} = 0,
	\end{equation*}
	confirming that $\mu_{\Phi^*}$ is a vanishing geometric Carleson measure.
\end{rem}

\begin{pf}[{\bf Proof of Theorem \ref{Dipon-vasu-p2-thm-05}}]
Combining Theorems 
\ref{Dipon-vasu-p2-thm-011}, 
\ref{Dipon-vasu-p2-thm-012} and \ref{Dipon-vasu-p2-thm-010}, we obtain Theorem
\ref{Dipon-vasu-p2-thm-05}. Indeed,
\[
\begin{aligned}
	C_\Phi \text{ is bounded on }\mathcal N_u
	&\iff \mu_{\Phi^*}\text{ satisfies the Carleson embedding property}\\
	&\iff \mu_{\Phi^*}\text{ satisfies the geometric Carleson condition}.
\end{aligned}
\]
Here, the symbol $'\iff'$ denotes the logical equivalence "if and only if."
Furthermore, under the standing assumption that supp($\mu_{\Phi^*})\subset\overline{\Phi(\mathbb C_0)}$, which is a compact subset of $\mathbb C_0$,
the measure $\mu_{\Phi^*}$ is automatically a vanishing Carleson
measure. Hence, by the vanishing version of the Carleson embedding
theorem,
\[
\begin{aligned}
	C_\Phi \text{ is compact on }\mathcal N_u
	&\iff \mu_{\Phi^*}\text{ satisfies the vanishing Carleson
		embedding property}\\
	&\iff \mu_{\Phi^*}\text{ satisfies the geometric vanishing
		Carleson condition}.
\end{aligned}
\]
This completes the proof of Theorem \ref{Dipon-vasu-p2-thm-05}.
\end{pf}

\section{Applications}\label{Dipon-vasu-p2-sec-5}

\subsection{Hyperbolic spherical derivatives} In this subsection, we introduce the notion of the hyperbolic spherical derivative, motivated by the hyperbolic distortion with respect to the Poincaré metric. Utilizing the Littlewood–Paley-type identity Theorem~\ref{Dipon-vasu-p2-thm-08} for the Nevanlinna class $\mathcal{N}_u$, we investigate the integral behavior of these derivatives via vertical limit functions. Recall that the classical spherical derivative of a meromorphic function $f$ on a planar domain $\Omega$ is defined by
\begin{equation*}
 f^\#_{Euc}(z) = \frac{|f'(z)|}{1+|f(z)|^2},
\end{equation*}
when $f(z)\ne\infty$. If $f(z)=\infty$, then $f^\#_{Euc}(z) = \lim_{\xi\to z}f^\#_{Euc}(\xi).$
This induces a spherical metric with constant positive curvature. Spherical derivative plays a central role in the theory of the classical Nevanlinna space $N(\mathbb{D})$. By the Littlewood–Paley-type identity for $N(\mathbb{D})$ (see for instance \cite[Lemma 1]{ChoaKim-PAMS-1997}), $f \in N(\mathbb{D})$ if and only if $f^\#_{Euc} \in L^2(\mathbb{D}, d\nu)$, where $d\nu(z) = (1-\vert{}z\vert{}^2) \, dA(z)$.

Similarly, comparing the spherical metric on the Riemann sphere $\widehat{\mathbb{C}}$ with the hyperbolic metric on the right half-plane $\mathbb{C}_0$ (where $\lambda_{\mathbb{C}_0}(z) = 1/\operatorname{Re}(z)$) yields the \textit{hyperbolic spherical derivative}
\begin{equation*}
	f^\#_{\mathrm{hyp}}(z) = 2 \operatorname{Re}(z) \frac{|f'(z)|}{1+|f(z)|^2} = 2 \operatorname{Re}(z) f^\#_{Euc}(z).
\end{equation*}
Given $f\in\mathcal{N}_u$ satisfying $\sigma_u(f)\le 0$, we assume
\begin{equation*}
	\vert\vert\vert f \vert\vert\vert_0:= \limsup_{\sigma\to 0^+}\lim_{T\to\infty}\frac{1}{2T}\int_{-T}^{T}\log(1+|f(\sigma+it)|^2)\,dt <\infty.
\end{equation*}
Note that 
\begin{equation*}
	\frac{1}{2}\vert\vert\vert f \vert\vert\vert_0 \le \vert\vert\vert f \vert\vert\vert_0 \le \log(2)+\vert\vert\vert f \vert\vert\vert_0.
\end{equation*}
The advantages of this new norm $\vert\vert\vert \cdot \vert\vert\vert_0$ lies on the integral representation \eqref{Dipon-vasu-p2-eqn-021} given below.
By applying the characterization from Theorem~\ref{Dipon-vasu-p2-thm-08} for the norm of $f \in \mathcal{N}_u$ in terms of vertical limit functions particularly for the subharmonic function $\log(1+\abs{f}^2)$, we obtain

	\begin{equation}\label{Dipon-vasu-p2-eqn-021}
	\vert\vert\vert f \vert\vert\vert_0\asymp \log(1+ |f_\chi(+\infty)|^2) + 4\int_{\mathbb{T}^{\infty}}\int_{0}^{\infty}\int_{\mathbb{R}}\frac{|f_\chi'(\sigma+it)|^2}{(1+|f_\chi(\sigma+it)|^2)^2}\,\sigma \,d\mu(t\,) \,d\sigma\, dm_\infty(\chi),
\end{equation}
where $\mu$ is a probability measure on $\mathbb{R}.$
An application of Fubini's theorem together with the definition of $f^\#_{hyp}$ shows that  
\begin{equation*}
	\int_{\mathbb{T}^{\infty}}\int_{0}^{\infty}\int_{\mathbb{R}}\frac{|f_\chi'(\sigma+it)|^2}{(1+|f_\chi(\sigma+it)|^2)^2}\,\sigma \,d\mu(t\,) \,d\sigma\, dm_\infty(\chi) = \frac{1}{4}\int_{\mathbb{T}^\infty}\int_{\mathbb{C}_0}\abs{(f^\#_{hyp})_\chi(z)}^2\,\frac{d\mu(t)\,d\sigma}{\sigma}\,dm_\infty(\chi).
\end{equation*}
Consequently 
\begin{equation*}
	f\in\mathcal{N}_u \iff f^\#_{hyp}\in L^2\left(\mathbb{T}^\infty\times\mathbb{C}_0\, ,\, dm_\infty(\chi)\frac{d\mu(t)\,d\sigma}{\sigma}\right).
\end{equation*}

In particular, if $f\in\mathcal{N}_u$ then $(f^\#_{hyp})_\chi\in L^2\left(\mathbb{C}_0\, ,\,\frac{d\mu(t)\,d\sigma}{\sigma}\right)$ for almost every $\chi\in\mathbb{T}^\infty.$ But the converse fails in general: a measurable function can be finite almost everywhere while its integral is infinite. To hold the converse, we need the sufficient condition 
\begin{equation*}
	{ess\,sup}_{\chi\in\mathbb{T}^\infty}\vert\vert\vert (f^\#_{hyp})_\chi \vert\vert\vert_{L^2(d\nu)} <\infty,\,\,\,\mbox{where}\,\,\,d\nu(t) = \frac{d\mu(t)\,d\sigma}{\sigma}.
\end{equation*}

\subsection{Invertible elements of the Nevanlinna class} This subsection is motivated by the work of Haldimann and Jarchow \cite{Haldimann-Studia-2001} on classical Privalov spaces $N^p(\mathbb{D})$ for $1 < p < \infty$, which consist of all holomorphic functions $f$ on the unit disk $\mathbb{D}$ satisfying
\begin{equation*}
	\sup_{0 \le r < 1} \int_{0}^{2\pi} \bigl(\log^+ |f(re^{i\theta})|\bigr)^p \, \frac{d\theta}{2\pi} < \infty.
\end{equation*}
The Privalov classes were introduced in the foundational work of Privalov \cite{Privalov-1941} in the context of the unit disk.
Equipped with the weighted measure $dm_\alpha(z) := (\alpha+1)(1-\vert{}z\vert{}^2)^\alpha \, dm(z)$ for $\alpha > -1$, where $dm$ denotes the normalized area measure on $\mathbb{D}$, Haldimann and Jarchow \cite{Haldimann-Studia-2001} introduced the weighted Privalov classes $N^p_\alpha(\mathbb{D})$. Recall that for $0 < p < \infty$, the weighted Bergman space $A^p_\alpha(\mathbb{D})$ consists of all functions $f \in \operatorname{Hol}(\mathbb{D})$ such that $f \in L^p(\mathbb{D}, dm_\alpha)$, where $\operatorname{Hol}(\mathbb{D})$ denotes the space of all holomorphic functions on $\mathbb{D}$. The class $N^p_\alpha(\mathbb{D})$ forms an $F$-algebra, often referred to as a Nevanlinna algebra (see \cite{Haldimann-Studia-2001} for details). In particular, they proved that $f \in N^p_\alpha(\mathbb{D})$ is invertible if and only if $f = e^g$ for some $g \in A^p_\alpha(\mathbb{D})$. We will provide an analogous result in our setting. Furthermore, it is straightforward to observe that the Nevanlinna class of Dirichlet series $\mathcal{N}_u$ (satisfying $\sigma_u(f) \le 0$) forms a unital $F$-algebra over $\mathbb{C}$. Its multiplicative identity is the constant Dirichlet series $\mathbf{1} := \sum_{n \ge 1} \delta_{n1} n^{-s}$, where $\delta_{ij}$ denotes the Kronecker delta.

\begin{thm}\label{Dipon-vasu-p2-thm-06}
	A Dirichlet series $f \in \mathcal{N}_u$ with $\sigma_u(f) \le 0$ is invertible, if $f = \exp(g)$ for some $g \in \mathcal{H}^1$.
\end{thm}
\begin{pf}
First, recall a lemma from \cite[Lemma~4.3]{Brevig-AdvMath-2021}, which asserts that if $f \in \mathcal{N}_u$ then
\begin{equation*}
	\limsup_{\sigma \to 0^+}\lim_{T \to \infty}\frac{1}{2T}\int_{-T}^{T}\bigl|\log|f(\sigma+it)|\bigr|\,dt < \infty.
\end{equation*}
Moreover converse holds easily.
Now suppose $f(s) = \exp(g(s))$ for some $g \in \mathcal{H}^1$. Then $\vert{}\log\vert{}f(s)\vert{}\vert{} \le \vert{}g(s)\vert{}$. Evaluating this inequality along vertical lines yields
\begin{equation*}
	\limsup_{\sigma \to 0^+}\lim_{T \to \infty}\frac{1}{2T}\int_{-T}^{T}\bigl|\log|f(\sigma+it)|\bigr|\,dt \le \limsup_{\sigma \to 0^+}\lim_{T \to \infty}\frac{1}{2T}\int_{-T}^{T}|g(\sigma+it)|\,dt.
\end{equation*}
Since $g \in \mathcal{H}^1$, the right-hand side is finite, which implies $f = \exp(g) \in \mathcal{N}_u$. By an identical calculation, $1/f = \exp(-g) \in \mathcal{N}_u$. Therefore, $f$ is invertible in $\mathcal{N}_u$.
\end{pf}

\begin{example}
    However, the converse of Theorem \ref{Dipon-vasu-p2-thm-06} does not hold in general.
	Consider the Dirichlet series
	\begin{equation*}
		g(s) = \frac{1+2^{-s}}{1-2^{-s}} = 1 + 2\sum_{n=1}^{\infty} 2^{-ns},
	\end{equation*}
	where $s\in\mathbb{C}_0$ and define $f(s) = \exp(g(s))$. Clearly $f$ is a Dirichlet series with $\sigma_u(f)\le 0.$ Set $r = 2^{-\sigma}$ and $\theta = t \log 2$. For $\sigma > 0$, we have $\vert{}2^{-s}\vert{} < 1$ and
	\begin{equation}\label{Dipon-vasu-p2-eqn-049}
		\operatorname{Re}(g(s)) = \operatorname{Re}\left(\frac{1+2^{-s}}{1-2^{-s}}\right) = \frac{1-r^2}{|1-r e^{-i\theta}|^2}.
	\end{equation}
	Thus $\operatorname{Re}( g(\sigma+it))$ is a Poisson kernel. By periodicity,
	\begin{equation}\label{Dipon-vasu-p2-eqn-048}
		\lim_{T \to \infty} \frac{1}{2T} \int_{-T}^{T} \operatorname{Re} (g(\sigma+it)) \, dt = 1, \quad \sigma > 0.
	\end{equation}
	Since $\operatorname{Re} (g(s))>0$, it follows that $\log^+\vert{}f(s)\vert{} = \log\vert{}f(s)\vert{} = \operatorname{Re} (g(s))$. Applying \eqref{Dipon-vasu-p2-eqn-048} yields
	\begin{equation*}
		\limsup_{\sigma \to 0^+} \lim_{T \to \infty} \frac{1}{2T} \int_{-T}^{T} \log^+|f(\sigma+it)| \, dt = 1.
	\end{equation*}
	Hence $f \in \mathcal{N}_u$. Furthermore, $\vert{}1/f(s)\vert{} = e^{-\operatorname{Re} g(s)} \le 1$, which implies $1/f \in \mathcal{H}^\infty \subset \mathcal{N}_u$. Consequently, $f$ is an invertible element of $\mathcal{N}_u$.
	We now claim, however, that $g \notin \mathcal{H}^1$. Indeed, $g$ is the periodic Dirichlet series obtained from the disk function $G(z) = (1+z)/(1-z)$ via the substitution $z = 2^{-s}$. The boundary values of $G$ satisfy $\vert{}G(e^{it})\vert{} = \left\vert{}\cot(t/2)\right\vert{} \asymp 1/\vert{}t\vert{}$ as $t \to 0$, from which it follows that
	\begin{equation*}
		\int_{-\pi}^{\pi} |G(e^{it})| \, dt = \infty.
	\end{equation*}
	Thus $G \notin H^1(\mathbb{D})$. By the Bohr correspondence for periodic Dirichlet series, we conclude that $g\notin\mathcal{H}^1.$ \qed
\end{example}

Since the converse of Theorem~\ref{Dipon-vasu-p2-thm-06} does not hold in general, it is natural to investigate the precise conditions under which this converse remains valid. This motivates the study of the Smirnov class of Dirichlet series, denoted by $\mathcal{N}_*$, which constitutes a distinguished subclass of the Nevanlinna class $\mathcal{N}_u$. Guo \textit{et al.} \cite{GuoZhou-AnnalFourier-2025} defined $\mathcal{N}_*$ as the completion of the space of Dirichlet polynomials under the metric \eqref{Dipon-vasu-p2-eqn-004}. Thus $\mathcal{N}_*$ can be viewed as the limiting space of the Hardy spaces $\mathcal{H}^p$ of Dirichlet series as $p \to 0^+$. Moreover, via the Bohr correspondence, there is a canonical isometric isomorphism from $\mathcal{N}_*$ onto the Smirnov class $N_*(\mathbb{D}^\infty_1)$ of infinitely many variables. Motivated by the construction of the Nevanlinna class of Dirichlet series, we introduce the Privalov classes of Dirichlet series. 

\begin{defn}\label{Dipon-vasu-p2-defn-03}A Dirichlet series $f$ with uniform convergence abscissa $\sigma_u(f) \le 0$ belongs to the \textit{Privalov class} $\mathcal{N}^p_u$ of Dirichlet series (for $1 < p < \infty$) if
	\begin{equation}
		\limsup_{\sigma \to 0^+} \lim_{T \to \infty} \frac{1}{2T} \int_{-T}^{T} \bigl(\log^+ |f(\sigma + it)|\bigr)^p \, dt < \infty,
	\end{equation}
where $\log^+(x) = \max\{\log x, 0\}$.
\end{defn}

The space $\mathcal{N}^p_u$ admits a metric structure analogous to that induced by $\Vert{}\cdot\Vert{}_0$. Furthermore, we have the proper inclusions
\begin{equation*}
	\bigcup_{p > 1} \mathcal{N}^p_u \subsetneq \mathcal{N}_* \subsetneq \mathcal{N}_u.
\end{equation*}
This inclusions follows directly from the classical results of Mochizuki \cite{Mochi-PAMS-1989} in the disk setting. The unit disk counterpart to the following  Theorem~\ref{Dipon-vasu-p2-thm-03} was established by Stoll in \cite{Stoll-Polon-1977}.

	
		
		

\begin{thm}\label{Dipon-vasu-p2-thm-03}
	Let $1 < p < \infty$. A Dirichlet series $f \in \mathcal{N}^p_u$ with $\sigma_u(f) \le 0$ is invertible if and only if $f = \exp(g)$ for some $g \in \mathcal{H}^p$.
\end{thm}
\begin{pf}
	One direction follows by an argument analogous to that in Theorem~\ref{Dipon-vasu-p2-thm-06}: if $f(s) = \exp(g(s))$ for some $g \in \mathcal{H}^p$, then $f \in \mathcal{N}^p_u$ with $\sigma_u(f) \le 0$ is invertible.
	
	Conversely, suppose that
	$f\in\mathcal N_u^p$, with $1<p<\infty$ and
	$\sigma_u(f)\le0$, is invertible in $\mathcal N_u^p$.
	Then $f$ and $1/f$ both belong to $\mathcal N_u^p$.
	In particular, $f$ is zero-free. Then there exists a holomorphic function $g$ on $\mathbb C_0$ such that
	\[
	f(s)=e^{g(s)}.
	\]
	Consequently, $\log|f(s)|=\operatorname{Re}(g(s))$ and hence
	\[
	|\operatorname{Re}(g(s))|
	=
	\log^+|f(s)|
	+
	\log^+\frac{1}{|f(s)|}.
	\]
	Using the inequality 
	\[
	(a+b)^p\le 2^{p-1}(a^p+b^p),
	\qquad a,b\ge0,
	\]
	we obtain
	\[
	|\operatorname{Re}(g(s))|^p
	\le
	2^{p-1}
	\left[
	\bigl(\log^+|f(s)|\bigr)^p
	+
	\left(\log^+\frac1{|f(s)|}\right)^p
	\right].
	\]
	Therefore, by the definition of $\mathcal N_u^p$, we have
	\[
	\limsup_{\sigma\to0^+}
	\lim_{T\to\infty}
	\frac1{2T}
	\int_{-T}^{T}
	|\operatorname{Re}(g(\sigma+it))|^p\,dt
	<\infty.
	\]
	
	Let $G=\mathcal Bg$ denote the Bohr lift of $g$. By the
	vertical-limit representation for Dirichlet series, the preceding
	estimate is equivalent to
	\[
	\sup_{\sigma>0}
	\int_{\mathbb T^\infty}
	|\operatorname{Re}(G_\sigma(\chi))|^p\,dm_\infty(\chi)
	<\infty,
	\]
	where
	\[
	G_\sigma(\chi)
	=
	G(2^{-\sigma}\chi_1,3^{-\sigma}\chi_2,\ldots),
	\]
	for almost every $\chi\in\mathbb{T}^\infty.$
	Since $1<p<\infty$, the Riesz projection on
	$L^p(\mathbb T^\infty)$ is bounded. Hence the harmonic conjugate
	of $\operatorname{Re}(G_\sigma)$ satisfies
	\[
	\|\operatorname{Im}(G_\sigma)\|_{L^p(\mathbb T^\infty)}
	\le
	C_p\|\operatorname{Re}(G_\sigma)\|_{L^p(\mathbb T^\infty)}
	+
	|\operatorname{Im}(G_\sigma(0))|.
	\]
	Since $G_\sigma(0)=g(+\infty)$ is independent of $\sigma$, we obtain
	\[
	\sup_{\sigma>0}
	\|G_\sigma\|_{L^p(\mathbb T^\infty)}
	<\infty.
	\]
	The standard Bohr characterization of the Hardy space
	$\mathcal H^p$ now gives $g\in\mathcal H^p.$
	This completes the proof.
\end{pf}
\begin{rem}
	Alternatively, the converse statement of Theorem \ref{Dipon-vasu-p2-thm-03} can be proved as follows. Since $\mathcal{N}^p_u \subset \mathcal{N}_*$ for $p > 1$, the function $f$ is invertible in the Smirnov class of Dirichlet series $\mathcal{N}_*$. Consequently, there exists an outer function $g$ such that $f(s) = \exp(g(s))$. Let $F = \mathcal{B}f$ and $G = \mathcal{B}g$ denote the Bohr lifts of $f$ and $g$ to the infinite torus $\mathbb{T}^\infty$. By \cite[Prop. 2.17]{GuoZhou-AnnalFourier-2025}, for every $\zeta \in \mathbb{D}^\infty_1$,
	\begin{equation*}
		\log|F(\zeta)| = \int_{\mathbb{T}^\infty} \mathbf{P}_\zeta \log|F^*| \, dm_\infty.
	\end{equation*}
	Since both $f$ and $1/f$ belong to $\mathcal{N}^p_u$, the functions $(\log^+\vert{}F^*\vert{})^p$ and $(\log^+\vert{}1/F^*\vert{})^p$ are integrable on $\mathbb{T}^\infty$. Noting that $\log^+\vert{}1/F^*\vert{} = \log^-\vert{}F^*\vert{}$, it follows that for $1 < p < \infty$,
	\begin{equation*}
		\log|F^*| \in L^p(\mathbb{T}^\infty).
	\end{equation*}
 Since $\log\vert{}F^*\vert{} \in L^p(\mathbb{T}^\infty)$ for $1 < p < \infty$, its Poisson integral defines the real part of $G(\zeta)$ for $\zeta \in \mathbb{D}^\infty_1$, with boundary values $\operatorname{Re}(G^*) = \log\vert{}F^*\vert{} \in L^p(\mathbb{T}^\infty)$ almost everywhere. By M. Riesz's theorem on conjugate functions \cite[p.~54]{Duren-H^p space-1970}, the boundary function satisfies $G^* \in L^p(\mathbb{T}^\infty)$, so $G \in \mathbb{H}^p(\mathbb{D}^\infty_1)$. Applying the inverse Bohr lift yields $g \in \mathcal{H}^p$.
\end{rem}

\vspace{2mm}

\noindent\textbf{Acknowledgment:} 
The second author gratefully acknowledges financial support from the CSIR–UGC, New Delhi, India.
 \vspace{1.5mm}

\noindent\textbf{Compliance of Ethical Standards:}
\vspace{1.5mm}\\
\noindent\textbf{Conflict of interest.} The authors declare that there is no conflict  of interest regarding the publication of this paper.
\vspace{1.5mm}

\noindent\textbf{Data availability statement.}  Data sharing is not applicable to this article as no datasets were generated or analyzed during the current study.\vspace{1.5mm}

\noindent\textbf{Authors contributions.} Both the authors have made equal contributions in reading, writing, and preparing the manuscript. \vspace{1.5mm}

\noindent\textbf{Disclosure of AI tools.} While preparing this manuscript, the authors used  ChatGPT as an auxiliary tool for a few mathematical discussions and to enhance the manuscript's organization, clarity, and language. Following its use, the authors reviewed and edited the content as necessary, accepting full responsibility for the final published article.

\end{document}